\documentclass{amsart}
\usepackage{amssymb, latexsym, amsfonts, amsthm, amsmath}
\usepackage{enumerate, epsfig, subfigure, color, graphicx}

\numberwithin{equation}{section}

\theoremstyle{plain}
\newtheorem{theorem}{Theorem}
\newtheorem{corollary}[theorem]{Corollary}
\newtheorem{lemma}[theorem]{Lemma}

\theoremstyle{definition}

\newtheorem{remark}[theorem]{Remark}
\newtheorem*{remark*}{Remark}
\newtheorem{observation}{Observation}

\numberwithin{theorem}{section}

\begin{document}
\title[Criteria for CP types]{Some criteria for circle packing types and combinatorial Gauss-Bonnet Theorem}
\author[B. Oh]{Byung-Geun Oh}
\address{Department of Mathematics Education, Hanyang University, 222 Wangsimni-ro, Seongdong-gu, Seoul 04763, Korea}
\email{bgoh@hanyang.ac.kr}

\date{\today}

\begin{abstract}
We investigate criteria for circle packing(CP) types of  disk triangulation graphs embedded into simply connected domains in $ \mathbb{C}$.
In particular, by studying combinatorial curvature and the combinatorial Gauss-Bonnet theorem involving boundary turns,
we show that a disk triangulation graph is CP parabolic if 
\[
 \sum_{n=1}^\infty \frac{1}{\sum_{j=0}^{n-1} (k_j +6)} = \infty,
 \]
where $k_n$ is the degree excess sequence defined by
\[
k_n = \sum_{v \in B_n} (\deg v - 6) 
\]
for  combinatorial balls $B_n$ of radius $n$ and centered at a fixed vertex. It is also shown that the simple random walk on a disk triangulation graph is recurrent if
\[
 \sum_{n=1}^\infty \frac{1}{\sum_{j=0}^{n-1} (k_j +6)+\sum_{j=0}^{n} (k_j +6)} = \infty.
 \]
These criteria are sharp, and generalize a conjecture by He and Schramm in their paper from 1995,
which was later proved by Repp in 2001. We also give several criteria for CP hyperbolicity, one of which generalizes a theorem of He and Schramm, and present a necessary and sufficient 
condition for CP types of layered circle packings generalizing and confirming a criterion given by Siders in 1998.
\end{abstract}

\maketitle

\section{Introduction}

Let $V$ be an index set, and suppose $P=\{ P_v : v \in V \}$ is an infinite circle packing of the plane; i.e., $P$  is a collection of infinitely many closed geometric disks 
$P_v \subset   \mathbb{C}$ such that if $u \ne v$,
then $P_u$ and $P_v$ have disjoint interiors. The \emph{contact graph}, which is also called the \emph{tangency graph} or the \emph{nerve}, of $P$ is the graph $G$ defined as follows: 
the vertex set of $G$ is nothing but the index set $V = V(G)$, 
and for distinct $u, v \in V$ an edge $[u, v]$ appears in the edge set $E= E(G)$  if and only if $P_u$ and $P_v$ intersect (Figure~\ref{F:cp}). Definitely $G$ is a planar graph describing the combinatorial pattern of $P$.

\begin{figure}[tbh]
\begin{center}
 \includegraphics[width=110mm]{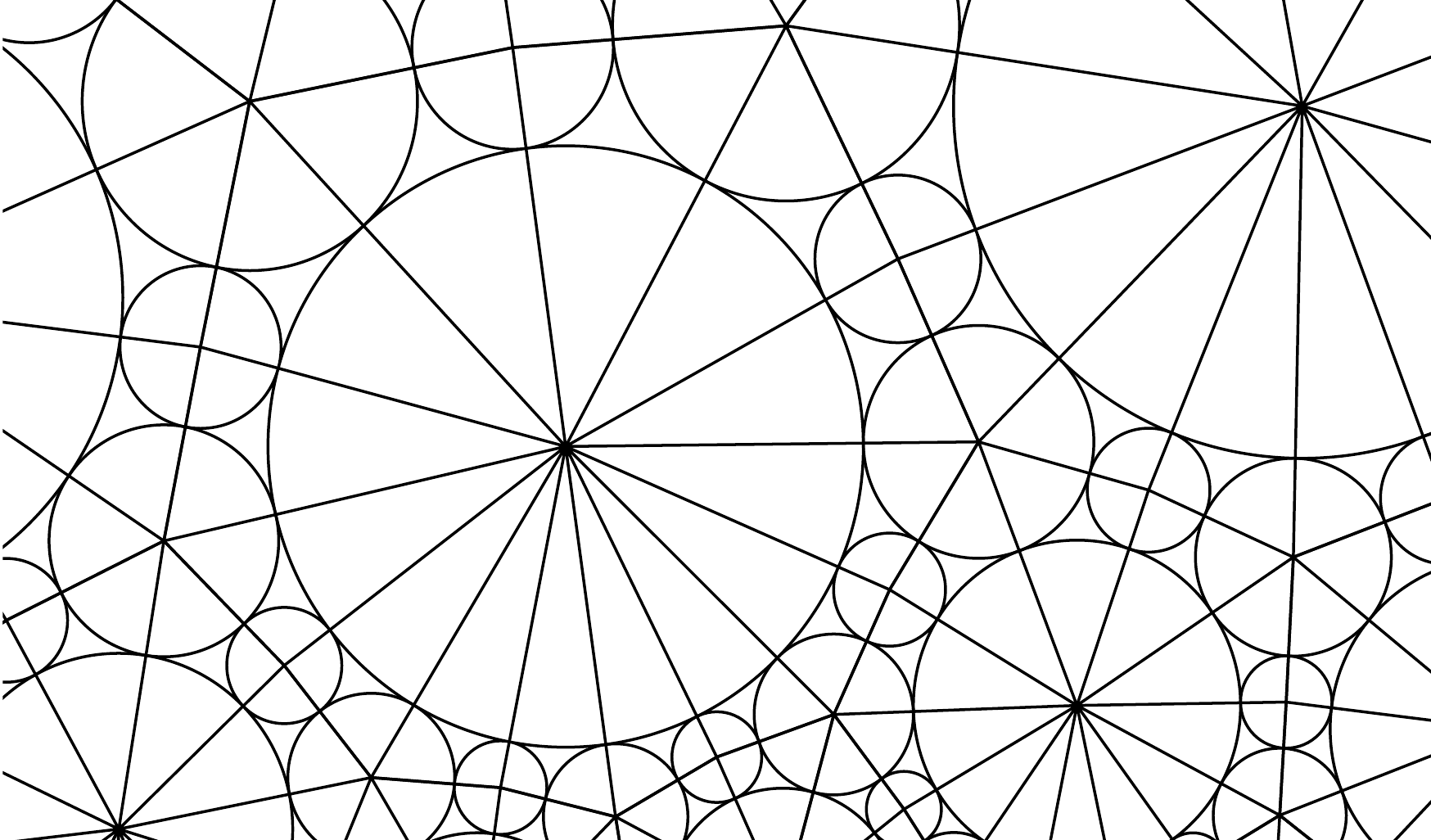}
 \caption{Part of a circle packing and its contact graph.}\label{F:cp}
\end{center}
\end{figure}

Now suppose that  a disk triangulation graph $G$ is given without any reference to circle packings.  (The term `disk triangulation graph'  will  always connote an \emph{infinite} graph 
throughout the paper. See Section~\ref{prelim} for details.) Then it is well known that
there exists a circle packing, unique up to M\"{o}bius transformations,  that  fills either the whole plane $  \mathbb{C}$ or the unit disk $\mathbb{D} := \{ z \in   \mathbb{C}: |z| <1 \}$,
and its contact graph is combinatorially equivalent to $G$ (cf.\;\cite[Corollary 0.5]{HS93}). Because M\"{o}bius transformations, or any other nonconstant analytic functions, 
cannot map the whole plane $  \mathbb{C}$ 
into the unit disk $\mathbb{D}$, the above statement  makes it possible 
to define the \emph{circle packing type} of a disk triangulation graph $G$;  i.e., we call $G$ \emph{circle packing parabolic} (CP parabolic for abbreviation) if there exists
a circle packing which fills the whole plane $ \mathbb{C}$ and whose contact graph is combinatorially equivalent to $G$, and we call $G$  \emph{circle packing hyperbolic} (CP hyperbolic)
if there exists a circle packing which fills the unit disk $\mathbb{D}$ and whose contact graph is combinatorially equivalent to $G$. The  question then arises as to under what condition
it is guaranteed that the given disk triangulation graph $G$ is CP parabolic or CP hyperbolic. There are many papers concerning this question 
\cite{BS91-To, Bo98, DW05, HS95, McC98, Oh15, Rep01, Si98, Wood09} 
(and more). Also see \cite{AHNR16, AHNR18, Asaf20} for probabilistic version of this problem. For those who want to study general theories about circle packings, we recommend
\cite{Ste03, Ste05}.  

\subsection{He and Schramm's criteria for circle packing types and the `gap'}

In \cite{HS95}, He and Schramm proved many criteria that determine whether a given disk triangulation graph is CP parabolic or CP hyperbolic. Among them, perhaps the following
is the most significant.

\begin{theorem}[He and Schramm]\label{HS1}
Let G be a disk triangulation graph. If $G$ is recurrent, then G is CP parabolic. 
Conversely, if   $G$ is transient and of bounded degree, then $G$ is CP hyperbolic. 
\end{theorem}

Note that a graph is called either \emph{recurrent} or \emph{transient} if the simple random walk on it is recurrent or transient, respectively, and \emph{of bounded degree}
if the degrees of vertices (the number of edges incident to each vertices) are bounded above. Theorem~\ref{HS1} was also proved independently by McCaughan \cite{McC98} for disk triangulation
graphs of bounded degree.

He and Schramm also proved in the same paper \cite{HS95} the following theorem.
\begin{theorem}[He and Schramm]\label{HS2}
Let $G$ be a disk triangulation graph, $S_0$ a finite nonempty subgraph of $G$, and $g: \mathbb{N} \to (0, \infty)$ a nondecreasing function.
For a subgraph $S$, define $dS$ as the set of vertices in $G$ not belonging to $S$ but having neighbors in $S$.
\begin{enumerate}[(a)]
\item If $G$ is CP parabolic and satisfies the perimetric inequality 
\[
|dS| \geq g(|S|)
\]
for every  connected finite subgraph $S \supset S_0$, where $| d S |$ denotes the cardinality of  the set  $d S$ and $|S|$ denotes the cardinality of the vertex set of $S$, then
\begin{equation}\label{E1}
\sum_{n=1}^\infty \frac{1}{g(n)^2} = \infty.
\end{equation}
\item If \eqref{E1} holds and the perimetric inequality 
\[
|d B_n| \leq g(|B_n|) 
\]
holds  for every $n =0,1, 2, \ldots$, where $B_n$ is the combinatorial ball of radius $n$ and centered at a fixed vertex, then $G$ is CP parabolic.
\end{enumerate}
\end{theorem}

Theorem~\ref{HS2}, or the aforementioned paper \cite{HS95} by He and Schramm, was published more than two and half decades ago, 
and since then there have been quite a few developments on the relationship 
between perimetric inequalities and type problems on graphs (about both CP types and recurrence/transience). For instance see~\cite[Chap.~6]{LP16} and references therein.

Using the two statements in Theorem~\ref{HS2}, He and Schramm proved  in \cite{HS95} the following theorem as well, which in fact has initiated our work in this paper.

\begin{theorem}[He and Schramm]\label{T:HSit}
Let $G$ be a disk triangulation graph.
\begin{enumerate}[(a)] 
\item If  at most finitely many vertices in $G$ have degrees greater than 6, then $G$ is CP parabolic and recurrent.
\item If we have the inequality
\begin{equation}\label{LAV}
\sup_{S_0} \inf_{S \supset S_0} \left( \frac{1}{|S|} \sum_{v \in S} \deg v \right) > 6,
\end{equation}
where $S_0$ and $S$ are nonempty connected finite subgraphs of $G$ and $\deg v$ denotes the degree of $v$, then $G$ is CP hyperbolic and transient.
\end{enumerate}
\end{theorem}

He and Schramm noted a wide gap between the two statements in Theorem~\ref{T:HSit}, and indicated that it would be interesting to narrow the gap. 
For this purpose they proposed the following conjecture, which was later proved by Repp \cite{Rep01}.

\begin{theorem}[Repp]\label{Rep}
Suppose $G$ is a disk triangulation graph of bounded degree. 
If the sequence
\begin{equation}\label{E:DES}
k_n = \sum_{v \in B_n} (\deg v - 6)
\end{equation}
is bounded, where $B_n$ is the combinatorial ball of radius $n$ and centered at a fixed vertex, then $G$ is CP parabolic.
\end{theorem}

The sequence $k_n$ will be called the \emph{degree excess} or \emph{valence excess} sequence as in \cite{Rep01}.  
The main purpose of this paper is  to narrow  the gap between  the statements in Theorem~\ref{T:HSit}  further; 
for example we wondered if we could replace the boundedness condition  of the sequence $k_n$ 
in  Theorem~\ref{Rep} by some other conditions, or if we could find a condition that is weaker than \eqref{LAV} but still guarantees CP hyperbolicity of a given disk triangulation graph. 

We first have the following theorem,  which generalizes Theorem~\ref{Rep}.

\begin{theorem}\label{T1}
 Suppose $G$ is a disk triangulation graph, and for each $n \in \mathbb{N}$ let $B_n$ be the combinatorial ball of radius $n$ and centered at a fixed vertex in $G$. 
 Let  $k_n$ be the degree excess sequence defined in \eqref{E:DES}, and let $a_n = \sum_{j=0}^{n-1} (k_j +6)$ for $n =1,2, \ldots$. Then $G$ is CP parabolic if 
 \begin{equation}\label{Reci}
 \sum_{n=1}^\infty \frac{1}{a_n} = \infty,
 \end{equation}
 and $G$ is recurrent if
  \begin{equation}\label{Reci2}
 \sum_{n=1}^\infty \frac{1}{a_n+a_{n+1}} = \infty.
 \end{equation}
 \end{theorem}

The sequence $a_n= \sum_{j=0}^{n-1}(k_j +6)$ in Theorem~\ref{T1} always satisfies the inequality $a_n \geq 3$ for every $n$, because it is a basic property satisfied by any degree excess sequence $k_n$ 
of disk triangulation graphs. See Remark~\ref{rm1} for details. An immediate corollary of Theorem~\ref{T1} is the following.

\begin{corollary}\label{C1}
Let $G$ and $B_n$ be as above. If $$k_n = \sum_{v \in B_n} (\deg v - 6) \leq c \ln n$$ for sufficiently large $n$, where $c$ is a fixed positive constant, then $G$ is CP parabolic and recurrent.
 \end{corollary}

If the degree excess sequence $k_n$  is bounded, 
then it is definitely less than $\ln n$ for sufficiently large $n$, so Corollary~\ref{C1} as well as Theorem~\ref{T1} implies Theorem~\ref{Rep}.
Also note that unlike  Theorem~\ref{Rep},  we do \emph{not} require $G$ to be \emph{of bounded degree}. 

\begin{proof}[Proof of Corollary~\ref{C1}]
Since $$\sum_{j=1}^{n} \ln j \leq n \ln n,$$ 
one can easily see that  
\[
a_n= \sum_{j=0}^{n-1} (k_j +6) \leq  (c+1) n \ln n
\]
for  sufficiently large $n$. Thus $a_n$ satisfies the assumptions in Theorem~\ref{T1} because $\sum_{n=2}^\infty 1/(n \ln n) = \infty$.
Corollary~\ref{C1} now follows from Theorem~\ref{T1}.
\end{proof}

We remark that the conditions \eqref{Reci} and \eqref{Reci2} in Theorem~\ref{T1} are sharp, which can be seen in the following theorem. 

\begin{theorem}\label{T:S}
Suppose $\{ k_n \}_{n=0}^\infty$ is a sequence of integers satisfying 
\begin{equation}\label{ICP}
  a_n = \sum_{j=0}^{n-1} (k_j +6) \geq 3
  \end{equation}
for every $n=1,2, \ldots$. If
$ \sum_{n=1}^\infty 1/a_n < \infty$
then there exists a CP hyperbolic disk triangulation graph $G$ such that 
\begin{equation}\label{DES2}
k_n = \sum_{v \in B_n} (\deg v - 6)
\end{equation}
for every $n=0,1,2, \ldots$, where $B_n$ is the combinatorial ball of radius $n$ and centered at a fixed vertex of $G$. Similarly if 
$\sum_{n=1}^\infty 1/(a_n+a_{n+1}) < \infty$
 then there exists a transient disk triangulation graph $G$ satisfying \eqref{DES2}.
 \end{theorem}
 
In Theorem~\ref{T:S} one cannot omit the condition \eqref{ICP} because it is a basic  property of disk triangulation graphs as mentioned before(cf.\ Remark~\ref{rm1}). 
Also note that if $k_n$ is the sequence such that $k_0=-3$, $k_{2n-1} =-3^n-3$  and $k_{2n} = 3^{n+1}-9$ for $n =1,2,3, \ldots$,  
we have $a_{2n-1} = \sum_{j=0}^{2n-2} (k_j +6) =3^n$ and $a_{2n} =3$ for $n =1,2, \ldots$. Thus in this case the example to be constructed in Theorem~\ref{T:S}
is transient since $ \sum 1/(a_n+a_{n+1}) < \infty$, while Theorem~\ref{T1} implies that it is CP parabolic because $\sum 1/a_n = \infty$.  

Let us briefly sketch our proof for Theorem~\ref{T1}. For a combinatorial ball $B_n$ we study the total combinatorial curvature $\kappa (B_n)$ of $B_n$ 
(see Sections~\ref{prelim} and \ref{S:CGB} for notation and terminology),
and  will check that it is a constant multiple of $k_n$. On the other hand,  by interpreting  combinatorially the total left turns(geodesic curvature)  made on the boundary of $B_n$,
we associate left turns with a formula involving the number of edges connecting $B_n$ to the rest of the graph. Then since the combinatorial Gauss-Bonnet theorem enables us to
convert information about  $k_n$ to information about the number of boundary edges, it will be possible to show 
 \[
 |S_n| \leq  \sum_{j=0}^{n-1} (k_j +6) = a_n \quad \mbox{and} \quad |E_n|\leq a_n + a_{n+1},
 \]
 where $S_n$ is the combinatorial sphere of radius $n$ and centered at a fixed vertex (i.e., the set of vertices with distance $n$ from a fixed vertex), 
 and $E_n$ is the set of edges connecting $B_n$ to the rest of the graph. Hence we can prove Theorem~\ref{T1} using the following versions of  Rodin and Sullivan's theorem 
\cite{RS87} (cf.\ \cite[p.\ 226]{Ste05}) and Nash-Williams' theorem \cite{NW59} (cf.\ \cite[p.\ 37]{LP16} or \cite[p.\ 55]{Soa}).
 
 \begin{theorem}[Rodin and Sullivan]\label{T:RoSull}
 A disk triangulation graph $G$ is CP parabolic if 
 $\sum_{n=1}^\infty 1/|S_n| = \infty$,
 where $S_n$ is the combinatorial sphere of radius $n$ and centered at a fixed vertex.
 \end{theorem}
 
 \begin{theorem}[Nash-Williams]\label{T:Nash}
 An infinite graph $G$ is recurrent if  
 $\sum_{n=1}^\infty 1/|E_n| = \infty$,
 where $E_n$ is the set of edges connecting $S_n$ to $S_{n+1}$. (Here $S_k$'s are combinatorial spheres as in Theorem~\ref{T:RoSull}.)
 \end{theorem}
 
We next discuss a criterion for CP hyperbolicity. Note that Theorem~\ref{T1} narrows the `gap'  from the parabolic cases,
so we need a theorem which fills the `gap' from the hyperbolic cases.

\begin{theorem}\label{T:Hyp}
 Suppose $G$ is a disk triangulation graph. If there exists a nondecreasing function  $g: \mathbb{N} \to (0, \infty)$  such that $\sum_{n=1}^\infty 1/g(n)^2 < \infty$ and 
\[
\sup_{S_0} \inf_{S \supset S_0} \frac{1}{g(|S|)} \sum_{v \in S} (\deg v - 6) >0,
\]
 where $S$ and $S_0$ are the same as in Theorem~\ref{T:HSit}, then $G$ is CP hyperbolic and transient.
\end{theorem}
 
By taking $g(x) = x^\alpha$ for $\alpha > 1/2$, it follows from Theorem~\ref{T:Hyp} that
 
 \begin{corollary}\label{C:useless}
 Let $G, S$, and $S_0$ be as in the previous theorem. If  there exists $\alpha > 1/2$ such that 
 \[
\sup_{S_0} \inf_{S \supset S_0} \frac{1}{|S|^\alpha} \sum_{v \in S} (\deg v - 6) >0,
\]
then $G$ is CP hyperbolic and transient. 
 \end{corollary}
Corollary~\ref{C:useless} can be considered a generalization of Theorem~\ref{T:HSit}(b),  because it is nothing but the case $\alpha=1$ in Corollary~\ref{C:useless}. 
See Section~\ref{S:doubt} for more explanation about this corollary, and some other hyperbolic criteria derived from previously known results.

The strategy for the proof of Theorem~\ref{T:Hyp} is similar to that of Theorem~\ref{T1}; i.e., we will check that the combinatorial Gauss-Bonnet theorem translates the information
about total curvature $\kappa(S)$ of an appropriate subgraph $S\subset G$ into information about the number of boundary edges of $S$, 
and that $\kappa(S)$ is a constant multiple of $\sum_{v \in S} (\deg v - 6)$. Thus in this procedure one can prove certain statements by transferring  conditions about
the quantity $\sum_{v \in S} (\deg v - 6)$ to already known perimetric conditions, 
which would be Rodin-Sullivan's and Nash-Williams' theorems for CP parabolicity and recurrence, respectively, in Theorem~\ref{T1}, 
and  we would use Theorem~\ref{HS2}(a) for the proof of Theorem~\ref{T:Hyp}.

\subsection{Layered circle packings and Siders' criterion}
A \emph{layered circle packing} is a circle packing such that all the circles in each layer have the same contact degrees; i.e., if $G$ is the contact graph of a circle packing, then
it is a layered circle packing if, for each  $n=0, 1,2, \ldots$, we have $\deg v = \deg w$ for $v, w \in S_n$. Here $S_n$ is the combinatorial sphere of radius $n$ and centered at 
a fixed vertex of $G$. See Figure~\ref{F:layered}. In particular, we are interested in the case that $\deg v \geq 6$ for every vertex $v \in V(G)$.
\begin{figure}[t]
\begin{center}
\includegraphics[width=90mm, height=90mm]{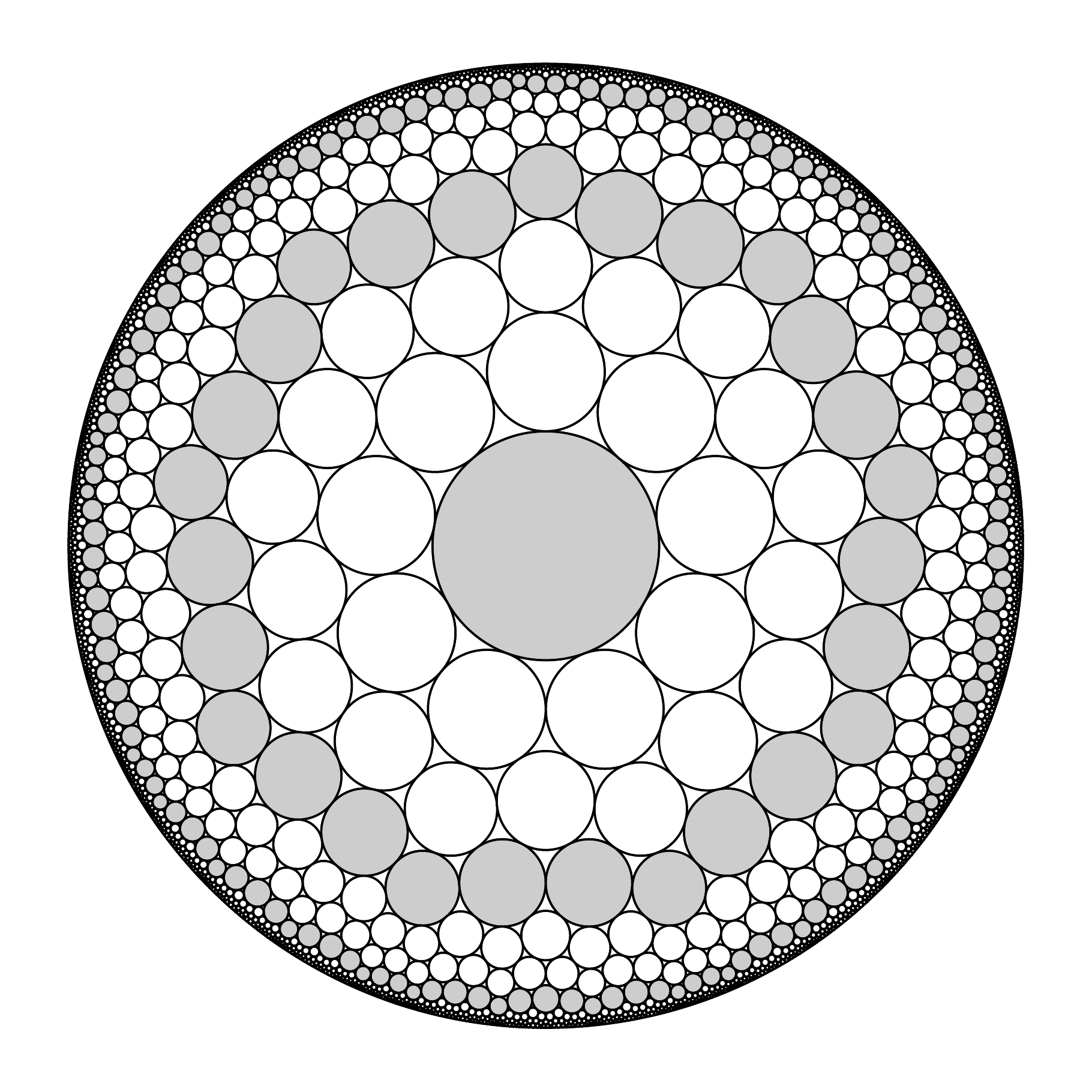}
\caption{A finite portion of a layered circle packing with $\deg v_0 =9$, $h_1 = h_2 =3$, $d_1 =1$ and $d_2 =2$. Circles of contact degrees greater than $6$ are shaded.}\label{F:layered}
\end{center}
\end{figure}

In order to describe the problem precisely, suppose $v_0$ is a vertex of $G$, $S_n$ the combinatorial sphere of radius $n$ and centered at $v_0$, and $(h_k, d_k)$
a sequence of ordered pairs in $\mathbb{N} \times \mathbb{N}$. Then our  assumption is the following:
$\deg v_0  \geq 6$, $\deg v = 6$ if $v \in S_1 \cup S_2 \cup \cdots \cup S_{h_1 -1}$, and $\deg v = d_1+6 \geq 7$ if $v \in S_{h_1}$. 
Similarly, for  $k \geq 2$ we assume that $\deg v = 6$ if $v \in S_m$ for $h_1 + h_2 + \cdots + h_{k-1} +1 \leq m \leq h_1 + h_2 + \cdots +h_{k-1} +h_{k} -1$, 
and $\deg v = d_k +6 \geq 7$ if $v \in S_{h_1 + h_2+ \cdots + h_k}$. In other words, we consider the circle packing such that for every $k\in \mathbb{N}$,
   $h_{k}-1$ layers consisting of degree $6$ vertices are followed by the $(h_1+h_2 + \cdots + h_{k})$-th layer consisting of degree $d_k +6 \geq 7$ vertices. 
Definitely  a layered circle packing is  determined by the sequence  of the pairs $(h_k, d_k)$, $k \geq 0$  (with $h_0=0$ and $d_0 = \deg v_0 -6$). 
  
 In \cite{Si98} Siders studied layered circle packings such that $d_k = 1$ for every $k=1,2, \ldots$, and proved the following result. 
 Note that in this case one only needs to consider the sequence $h_k$ instead of the sequence of pairs $(h_k, d_k)$. 
 
\begin{theorem}[Siders]\label{T:siders}
Suppose $G$ is a disk triangulation graph associated with a layered circle packing determined by the sequence $h_k$ of natural numbers. Then $G$ is CP hyperbolic if
\begin{equation}\label{E:LCP-P}
 \sum_{n=2}^\infty \frac{\ln h_n}{\prod_{k=1}^{n-1} h_k } < \infty,
 \end{equation}
 and CP parabolic if
 \[ 
 \sum_{n=2}^\infty \frac{\ln h_n }{2^n \prod_{k=1}^{n-1} h_k} = \infty.
 \]
\end{theorem}

There is in fact no explicit statement in \cite{Si98}, where the criteria above were formulated in the course of  proof/explanation. Also remark that the formulae in \cite{Si98} were
written using the number of degree $6$ layers (i.e., $h_k -1$ in our notation) instead of $h_k$. Because $h_k$ was, however, implicitly assumed very large
in \cite{Si98}, we believe that the statements written in  Theorem~\ref{T:siders} are not different from what Siders formulated in \cite{Si98}.
It was also asserted in the abstract of \cite{Si98} that a layered circle packing is parabolic or hyperbolic according as the sum $\sum \ln h_k /c_{k-1}$  diverges or converges,
respectively, where $c_k$ is the number of vertices in $S_{h_1 + \cdots + h_k}$. Unfortunately, however,
we could not find  in \cite{Si98} any rigorous proof for this assertion, at least for the hyperbolic case. In this context we provide the next theorem for supplements of Siders' result.

\begin{theorem}\label{T:LCP}
 Suppose $G$ is a disk triangulation graph associated with a layered circle packing determined by the sequence $(h_k, d_k) \in \mathbb{N} \times \mathbb{N}$.
 Then $G$ is CP parabolic if and only if
 \begin{equation}\label{E:sumeq}
 \sum_{n=2}^\infty \frac{\ln h_n}{d_{n-1} c_{n-1}} = \infty,
  \end{equation}
 where $c_n$ is a sequence of integers defined by
 \begin{equation}\label{E:LCP}
 \begin{cases} 
 c_1 = h_1,  \\
 c_n = c_{n-1} + h_n ( 1+d_1 c_1 + d_2 c_2 + \cdots + d_{n-1} c_{n-1}) \   \mbox{ for } n \geq 2.
 \end{cases}
 \end{equation}
 \end{theorem}

We will show in Section~\ref{S:LCP} that $c_n = (\deg v_0) \cdot |S_{h_1 + h_2 + \cdots + h_n}|$, so Theorem~\ref{T:LCP} is not much different from what Siders claimed in 
the abstract of \cite{Si98}. Moreover, there is  some improvement  in Theorem~\ref{T:LCP} because the 
exact recurrence relation for $c_n$ is given, and more importantly, the criterion is proved without any boundedness assumption for $d_n$.
Although it seems almost impossible to solve the sequence $c_n$ explicitly, the recurrence relation \eqref{E:LCP} is useful because it helps us to prove the following corollary, which claims that
the appropriate quantity we have to consider for CP types of layered circle packings is  essentially the left-hand side of \eqref{E:LCP-P}.

\begin{corollary}\label{C:LCP}
Suppose $G$ is a disk triangulation graph associated with a layered circle packing determined by the sequence $(h_k, d_k)$. Then $G$ is CP hyperbolic if 
\[
\sum_{n=2}^\infty \frac{\ln h_n}{\prod_{k=1}^{n-1} d_k h_k } < \infty.
 \]
 Furthermore, if $\sum_{n=1}^\infty 1/ h_{n}$ is convergent and
 \[
 \sum_{n=2}^\infty \frac{\ln h_n}{ \prod_{k=1}^{n-1} d_k h_k} = \infty,
 \]
 then $G$ is CP parabolic.
 \end{corollary}
 \begin{proof}
Since $c_n \geq (1+ h_n d_{n-1}) c_{n-1} \geq h_n d_{n-1} c_{n-1}$ for every $n \geq 2$, the hyperbolic part easily comes from Theorem~\ref{T:LCP}.
Now suppose $h_n$ grows rapidly so that the series $\sum_{n=1}^\infty 1/ h_{n}$ converges.  From \eqref{E:LCP} we obtain that
\begin{equation}\label{c_n}
\begin{aligned}
c_n & = c_{n-1}+ h_n d_{n-1}   c_{n-1} + h_n ( 1+ d_1 c_1 + \cdots + d_{n-2} c_{n-2}) \\
       & = c_{n-1} + h_n d_{n-1}  c_{n-1} + \frac{h_n}{h_{n-1}} (c_{n-1} - c_{n-2})\\
       & = (1+ \alpha_n) h_n d_{n-1} c_{n-1},
\end{aligned}\end{equation}
where
\[
0 \leq \alpha_n = \frac{1}{h_n d_{n-1}} + \frac{1}{h_{n-1} d_{n-1}} \cdot \frac{ c_{n-1} - c_{n-2}}{c_{n-1}} \leq \frac{1}{h_n} + \frac{1}{h_{n-1}}
\]
for $n \geq 3$, because $d_k \geq 1$ for every $k$ and $c_k$ is an increasing positive sequence. Then since the series $\sum 1/ h_{n}$ converges,  both  the  series  $\sum \alpha_n$
and the infinite product $\prod_{n=3}^\infty ( 1+ \alpha_n)$  converge as well. Thus we get from \eqref{c_n} that
$d_{n} c_{n} \leq \beta \cdot \prod_{k=1}^{n} d_k h_k$ for some absolute positive constant $\beta$. Now the parabolic part of the corollary follows from Theorem~\ref{T:LCP},
and this completes the proof.
\end{proof} 

The condition $\sum 1/h_n < \infty$, or some other conditions similar to it, cannot be omitted from the parabolic part of Corollary~\ref{C:LCP}. 
To see this, let $d_n =1$ for all $n$, and $h_n$ be a sequence of natural numbers with a subsequence $h_{n_k}\geq 2$ which appears scarcely and grows rapidly. 
We also assume that  $h_n =2$ unless $n = n_k$ for some $k$. Then definitely both the infinite sum $\sum 1/h_n$ and the infinite product  $\prod_{n} (1+ \alpha_n)$  diverge
because $\alpha_n \geq 1/h_n$, where $\alpha_n$ is the sequence in \eqref{c_n}. 
This means that $c_{n} \gg \prod_{k=1}^n h_k$ for sufficiently large $n$, hence we can manipulate $h_{n_k}$ so that
\[
 k \cdot \prod_{j=1}^{n_k -1} h_j \leq  c_{n_k-1} \quad \mbox{and} \quad \frac{1}{k} \leq \frac{\ln h_{n_k}}{ \prod_{j=1}^{n_k -1} h_j} \leq \frac{2}{k}.
\]
Now Theorem~\ref{T:LCP} implies that the graph $G$  is CP hyperbolic because 
\begin{align*}
 \sum_{n=2}^\infty \frac{\ln h_n}{c_{n-1}} & = \sum_{n \ne n_k} \frac{\ln h_n}{c_{n-1}} +  \sum_{k=1}^\infty \frac{\ln h_{n_k}}{c_{n_{k}-1}}\\
 & \leq \sum_{n=1}^\infty \frac{\ln 2}{ 2^{n-1}} + \sum_{k=1}^\infty \frac{2}{k^2} < \infty,
  \end{align*}
where we used the inequality $c_{n-1} \geq \prod_{j=1}^{n -1} h_j \geq 2^{n-1}$, but definitely we have 
 \[
 \sum_{n=2}^\infty \frac{\ln h_n}{ \prod_{j=1}^{n-1}  h_j} \geq \sum_{k=1}^\infty \frac{\ln h_{n_k}}{ \prod_{j=1}^{n_k-1} h_j} \geq 
 \sum_{k=1}^\infty \frac{1}{k} = \infty.
 \]

\medskip

We will prove Theorem~\ref{T:LCP} in Section~\ref{S:LCP}, where it is first proved that the left-hand side of \eqref{E:sumeq} converges if and only if the sum
$\sum 1/|S_n|$ converges. Therefore one direction of Theorem~\ref{T:LCP} comes from Theorem~\ref{T:RoSull} (Rodin and Sullivan's Theorem).
We need some extra works for the other direction, however, since we do not assume that vertex degrees are bounded. Hence we will modify Siders' method and show that
the vertex extremal distance (cf.\ Section \ref{S:VEL}) of the graph is finite when the left-hand side of \eqref{E:sumeq} is finite, with some aids of Schamm's work on
square tilings \cite{Sch93}. Because vertex extremal distance is finite if and only if the considered disk triangulation graph $G$ is CP hyperbolic, this will complete the proof. 

\subsection{Structure of the paper}
This paper is organized as follows. In Section~\ref{prelim} we explain some terminology and notation used in this paper, among which there are several unusual notation such 
as the region $D(S)$ corresponding to a subgraph $S$, and various boundaries $b S$, $d S$, and $\partial S$ of $S$. In Section~\ref{S:CGB} the concept of
combinatorial curvatures and two versions of the combinatorial Gauss-Bonnet theorem will be introduced. Though there is no serious proof in  Section~\ref{S:CGB},
the combinatorial Gauss-Bonnet theorem will play the most important role in this paper, and in \cite{Oh20+}.
Our main results, Theorem~\ref{T1} and Theorem~\ref{T:Hyp}, will be proved in Section~\ref{S:T1} and Section~\ref{S:T2}, respectively. The pros and cons of
Theorem~\ref{T:Hyp} (or Corollary~\ref{C:useless}) are explained in Section~\ref{S:doubt}, where two more criteria for CP hyperbolicity will be deduced from some
previously known results. Experts may skip Section~\ref{S:VEL}, because the materials in this section are vertex extremal length, square tilings, 
and some previously known results. In Section~\ref{S:des} we construct an example for Theorem~\ref{T:S}. The paper finishes in Section~\ref{S:LCP}, where we study
layered circle packings and prove Theorem~\ref{T:LCP}.  

\section{Preliminaries}\label{prelim}
A graph $G$ is a pair of the vertex set $V= V(G)$ and the edge set $E = E(G) \subset V \times V$, and because we only consider undirected graphs, we always assume
that an edge $[v, w] \in E$ is the same as the edge $[w, v]$, where $v, w \in V$. A graph is called finite or infinite according as the cardinality of the vertex set. 
A \emph{planar} graph is a graph that is already embedded into the plane $\mathbb{C}$ (if it is infinite) or the Riemann sphere $\hat{\mathbb{C}} = \mathbb{C} \cup \{ \infty \}$
(if it is finite)  \emph{locally finitely},  which means that every compact set in $\mathbb{C}$ or $\hat{\mathbb{C}}$, respectively, intersects only finitely many vertices and edges of the graph.  
Strictly speaking, a planar graph is different from its embedded graph, but we do not distinguish them and use the term `planar graph' for embedded graphs. 

Let $G$ be a planar graph embedded into $D$, where $D$ is either $\mathbb{C}$ or $\hat{\mathbb{C}}$. The closure of a component of $D \setminus G$ is called a \emph{face}
of $G$, and we denote by $F = F(G)$ the set of faces of $G$. Since each planar graph is determined by the triple $(V, E, F)$, we identify it with $G$ and use the notation
$G=(V, E, F)$.

Note that every face of $G$ is a closed set with respect to the Euclidean topology. Similarly we will always regard each 
vertex and edge of $G$ as a closed set.  Two objects
in $V \cup E \cup F$ are called \emph{incident} if one is a proper subset of the other, and two distinct vertices $v, w \in V$ are called \emph{neighbors}
if they are incident to  the same edge, that is, if $[v, w] \in E$. A \emph{path}  is a sequence $[\ldots, v_{n-1}, v_{n}, v_{n+1}, \ldots ]$ of vertices in $V$ such that
$[v_n, v_{n+1}] \in E$ for all $n$, where the sequence could be finite, or infinite in one direction or both directions. 
If it is finite,  the \emph{length} of the path is the number of edges traversed by those moving from the initial vertex to the terminal vertex; i.e., the length of the path $[v_0, v_1, \ldots, v_n]$ is $n$. 
 A \emph{cycle} is a finite path $[v_0, v_1, \ldots, v_n]$ such that $v_0 = v_n$,
and a  path is called \emph{simple} if every vertex appears in the sequence at most once except the case that the path is a cycle and the only repetition is
the initial and terminal vertices. A \emph{facial walk} of a face $f \in F$ is the cycle (or a union of cycles) 
which is obtained by walking along the topological boundary of $f$ in the positive direction (Figure~\ref{F:facialwalk}).

\begin{figure}[t]
\centerline{
\hfill \subfigure[a triangular face]{\scalebox{0.8}{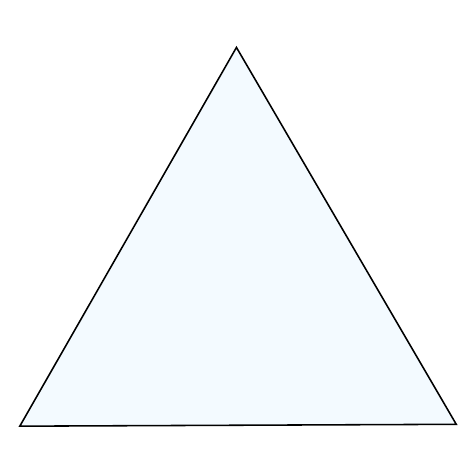}} \hfill
\subfigure[a face in a non tessellation graph]{\scalebox{0.8}{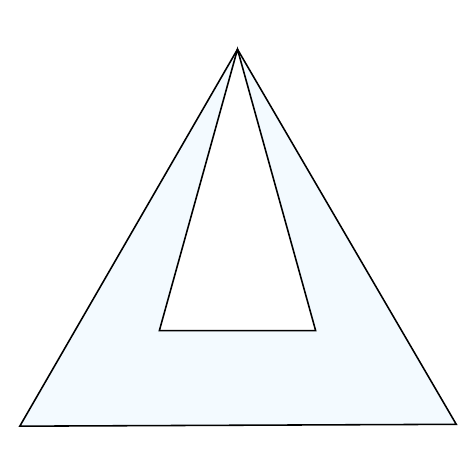}} \hfill
\subfigure[a face in a disconnected graph]{\scalebox{0.8}{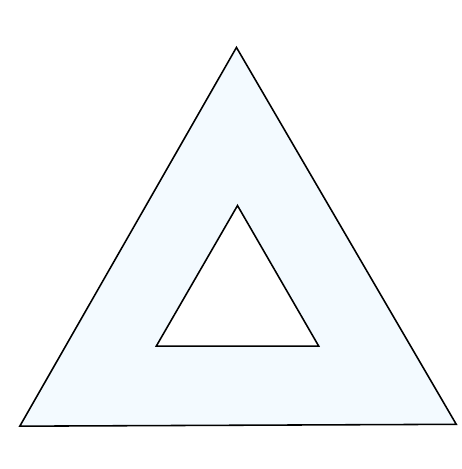}} \hfill
}
\caption{Some faces of planar graphs. The facial walks are, from the left, $[ v_1, v_2, v_3, v_1]$, $[v_1, v_2, v_3, v_1, v_5, v_4, v_1]$, and $[v_1, v_2, v_3, v_1] \cup [w_1, w_3, w_2, w_1]$.
(The faces in (b) and (c) will not be considered in subsequent sections.)}\label{F:facialwalk}
\end{figure}

In graph theory a \emph{path} usually refers to a \emph{simple path}, and the term \emph{walk} is used for an object that is essentially the same as a \emph{path}  
in our definition (cf.\ \cite{BoMur, Dies}). But we use the term \emph{path} even though it has some repetitions of vertices, 
because  we believe that this is a convention for the community of  circle packing theory. 
On the other hand, we will use the word `walk' for either a path or a union of paths, such as a facial walk defined  above.
For cycles, we do not require  them to be simple, and we even allow cycles of length zero or two; i.e., a cycle could be of the form $[v_0]$ or $[v_0, v_1, v_2=v_0]$,
respectively. (If $G$ has a self loop $[v, v]$ then a cycle of length one is also allowed, but we will not be bothered by this `bad' case since our works will stay mostly on triangulations, or
at least tessellations, defined below.)

The \emph{degree} of a vertex $v \in V$ is the number of  edges incident to $v$, for which we use the notation $\deg v$.
Similarly the \emph{degree} or \emph{girth} of a face $f \in F$ is the number of edges incident to $f$, denoted by $\deg f$. In fact,  for the degrees
we have to count the number of incident edges using \emph{multiplicities}; i.e., for example if a vertex $v \in V$ is incident to a self loop $e=[v, v] \in E$, then $e$ should be counted twice in
the computation of $\deg v$. (As mentioned in the previous paragraph, this bad case will not appear in this paper except the current section.)

A graph is  called \emph{connected} if every two vertices can be joined by a path, and  \emph{simple} if it has no self loop nor multiple edges; i.e., every edge is
incident to two distinct vertices and there is at most one edge connecting two vertices. Following \cite{BP01, BP06}, we call a connected simple planar graph a \emph{tessellation} or \emph{tiling}
if the following three conditions hold: (a) every edge is incident to two distinct faces; (b) every two faces are either disjoint or intersect in one vertex or one edge; and (c) every face is homeomorphic to a closed disk (i.e., the facial walk of  each face is a simple cycle of length at least three--the face in Figure~\ref{F:facialwalk}(b) is not allowed for tessellations). 
Note that if the given graph is a tessellation,  then we have $\deg v \geq 3$  for every vertex $v$ and $\deg f \geq 3$ for every face $f$. Furthermore we have $\deg v < \infty$
for every $v \in V$ since we have assumed that a planar graph is already embedded into $\mathbb{C}$ (or $\hat{\mathbb{C}}$) locally finitely. 
Finally with the term \emph{disk  triangulation graph}, we mean a tessellation in $\mathbb{C}$ with infinitely many vertices such that $\deg f = 3$ for every face $f$.

 For $v, w \in V$, the combinatorial distance between $v$ and $w$, denoted by $d(v, w)$,
is the minimum of the lengths of paths connecting $v$ and $w$. If we need to consider two or more graphs simultaneously,  we will use the notation $d_G (v, w)$ instead of $d(v, w)$.
 For $n \in \mathbb{N}$ and $v_0 \in V$, the combinatorial ball $B_n(v_0)$ (or the combinatorial sphere $S_n(v_0)$) 
of radius $n$ and centered at $v_0$ is defined by the set of vertices whose combinatorial distance from $v_0$ is at most $n$ (or is equal to $n$, respectively). If there is no confusion,
the notation $B_n$ and $S_n$ will be used instead of $B_n (v_0)$ and $S_n (v_0)$ as we already did in the introduction. We also have the convention that $B_0 = S_0 = \{ v_0 \}$.

Now let $G=(V, E, F)$ be a connected (infinite) planar graph. A planar graph $S$ is called a \emph{subgraph} of $G$ if $V(S) \subset V$ and $E(S) \subset E$. In this case we assume that
the face set $F(S)$ of $S$ is the \emph{subset} of $F$ which is defined as follows: $f \in F$ belongs to $F(S)$ if and only if  all the edges incident to $f$ are in $E(S)$. This notation might be confusing,
since  $F(S)$ is in fact the intersection of $F$ and the set of closures of components of $\mathbb{C} \setminus S$ (i.e., the intersection of $F$ and the face set of $S$, where $S$ is considered 
a planar graph by itself). We use the notation $S \subset G$ if $S$ is a subgraph of $G$. We call $S \subset G$ an \emph{induced} subgraph if it is induced by its vertex set; that is,
$S$ is an induced subgraph if we have $[v, w] \in E(S)$ whenever $v, w \in V(S)$ and $[v, w] \in E = E(G)$. 

If $X$ is a subset of $V$, we will treat $X$  not only as a subset of $V$ but also as an induced subgraph of $G$. For example we have defined $B_n (v_0)$ as the set of vertices  whose combinatorial
distance from $v_0$ is at most $n$, but we will also consider $B_n$  the subgraph of $G$ induced by the vertices in $B_n$. Conversely if $S$ is an induced subgraph of $G$, then
we do not distinguish $S$ from $V(S)$ and use the notation $|S|$ for $|V(S)|$, where $| \cdot |$ is the cardinality of the given set. (This convention might be used 
even when $S$ is not induced.) Also we will treat a path $[\ldots, v_{n-1}, v_{n}, v_{n+1}, \ldots ]$  as a subgraph of $G$ with vertices $v_i$ and edges $[v_i, v_{i+1}]$.

Let $S$ be a subgraph of $G$. We define $d S$, the vertex boundary of $S$, as the set of vertices in $V \setminus V(S)$ which have neighbors in $V(S)$. We also define
$\partial S$, the edge boundary of $S$, as the set of edges in $E$ that are incident to one vertex in $V(S)$ and another vertex in $V \setminus V(S)$. Finally we define $b S$ as the 
boundary walk of $S$; i.e., $b S$ is a union of cycles which are obtained as one walks along the topological boundary of 
\begin{equation}\label{D(S)}
D(S) := S \cup  \left(\bigcup_{f \in F(S)} f \right)=   \left( \bigcup_{v \in V(S)} v  \right) \cup \left( \bigcup_{e \in E(S)} e \right) \cup  \left( \bigcup_{f \in F(S)} f  \right)
\end{equation}
 in the positive direction. 
Then one can easily see that if $S$ is connected and $\mathbb{C} \setminus D(S)$ has $m$ components,  $b S$ can be written as a union of $m$ cycles,
each of which corresponds to the topological boundary of a component of $\mathbb{C} \setminus D(S)$ (Figure~\ref{F:D(S)}). Note that $b S$ is considered a subgraph of $G$,
so notation like $V(b S)$ or $E(b S)$ will be used in subsequent sections.

\begin{figure}[t]
\begin{center}
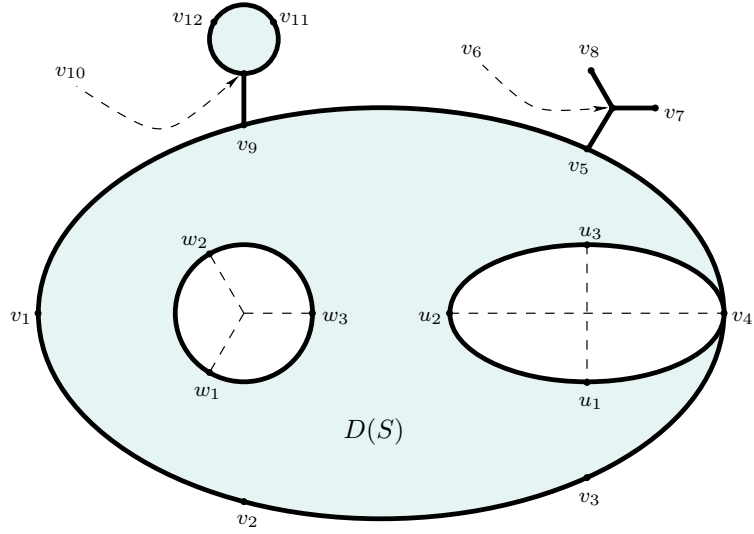
\caption{The region $D(S)$ corresponding to a subgraph $S$. In this figure
the boundary walk $bS$  can be written as $\Gamma_1 \cup \Gamma_2 \cup \Gamma_3$, where 
$\Gamma_1=[v_1, v_2, \ldots, v_7, v_6, v_8, v_6, v_5, v_9, \ldots, v_{12}, v_{10}, v_9, v_1]$, $\Gamma_2 = [u_1,u_2,u_3, v_4, u_1]$,
and $\Gamma_3 = [w_1,w_2,w_3,w_1]$, each of which corresponds to the topological boundary of a component of $\mathbb{C} \setminus D(S)$.}\label{F:D(S)}
\end{center}
\end{figure}

\section{Combinatorial Gauss-Bonnet Theorem} \label{S:CGB}
Let  $G= (V, E, F)$ be a tessellation such that $\deg f < \infty$ for every $f \in F$. 
For $v \in V$, we define the \emph{combinatorial curvature} at $v$  by
\[
\kappa (v) := 1 - \frac{\deg v}{2} + \sum_{f \sim v} \frac{1}{\deg f},
\]
where the summation is over all the faces incident to $v$. 

To explain the meaning of combinatorial curvature, we associate each face $f \in F$ with a  regular ($\deg f$)-gon of side lengths one, and 
paste them along sides exactly by the way that the corresponding faces are lying on $\mathbb{C} \supset G$. Then the resulting surface, which we denote by $\mathcal{S}_G$,
becomes so-called a \emph{polyhedral surface}, which naturally contains $G$ and is locally Euclidean everywhere except  the singularities at vertices of $G$. 
(See \cite{AZ, Res} for more about polyhedral surfaces.) Moreover,
one can easily check that the \emph{total angle} $T(v)$ at each vertex $v \in V \subset \mathcal{S}_G$ is $\sum_{f \sim v} (\pi - 2\pi/ \deg f)$, hence the atomic curvature at $v$ is 
\begin{equation}\label{curvertex}
\omega \bigl( \{ v \} \bigr) = 2 \pi - T(v) = 2 \pi - \sum_{f \sim v} \left( \pi - \frac{2 \pi}{\deg f} \right)  = 2 \pi \cdot \kappa(v).
\end{equation}
Here $\omega$ denotes the integral curvature defined on each Borel set in $\mathcal{S}_G$.
Thus the combinatorial curvature is just the usual integral curvature of the surface $\mathcal{S}_G$, but it is normalized so that $2 \pi$ becomes $1$. (The normalization was done 
because we just deal with combinatorial properties and do not want to have $2 \pi$ in every formula.) 

It is not clear when the concept of combinatorial curvature emerged. For 3-polytopes,  it seems to have been considered even by Descartes (without the normalization) \cite{Fed82}. 
Moreover, the quantity $-2 \kappa (v)$, called \emph{excess}, appears in a book by Nevanlinna  \cite[p.\ 311]{Nev},  which was originally published in the early 20th century. 
However, it is rather recently that combinatorial curvature has been   studied extensively, especially related to planar graphs. See for example  
\cite{BP01, BP06, Chen09, CC08,  DeMo07, Hig01, HJ15, HuS19,  Keller11, Oh14, Oh17, OS16, Sto76,  Woess98,  Zuk97} and the references therein.

When studying geometry of 2-surfaces, 
perhaps the Gauss-Bonnet Theorem is one of the most famous and useful theorems  related to curvature. It is not different in the study of geometric properties of planar graphs. 
For example, the following form of the Gauss-Bonnet Theorem can be found in many literature \cite[and more]{Chen09, DeMo07, Oh17}.

\begin{theorem}[Combinatorial Gauss-Bonnet Theorem-Basic Form]
Suppose $G$ is a connected simple finite graph embedded into a compact 2-manifold $\Omega$ without boundary. Then we have
\[
\kappa(G) = \sum_{v \in V(G)} \kappa(v) = \chi(\Omega),
\]
where $\chi(\Omega)$ denotes the Euler characteristic of $\Omega$.
\end{theorem}

The Gauss-Bonnet Theorem we  need, however, is a little bit more complicated and delicate than the basic form. In fact, we need those involving left turns(geodesic curvature) 
of the topological boundary of each subgraph. To explain it, let $G$ be an infinite tessellation in $\mathbb{C}$, 
and suppose $\Gamma = [v_0, v_1, v_2, \ldots, v_{n-1}, v_n=v_0]$ is a cycle in $G$ of length $n \geq 3$. 
For each $k=1,2, \ldots, n$, let $f_1^{(k)}, f_2^{(k)}, \ldots, f_{s_k}^{(k)}$ be the faces of $G$ incident to $v_k$
and \emph{lying on the right} of the path $[v_{k-1}, v_k, v_{k+1}]$, where we interpret $v_{n+1} = v_1$ if $k=n$. See Figure~\ref{F:leftturn}.
Then the \emph{outer left turn} occurred at $v_k$ is defined by
\begin{equation}\label{outerLT}
\tau_{o} (v_k) =  \sum_{j=1}^{s_k} \left( \frac{1}{2} - \frac{1}{\deg f_j^{(k)}} \right) - \frac{1}{2},
\end{equation}
and the outer left turn of $\Gamma$ is defined by the sum $\tau_{o} (\Gamma) = \sum_{k=1}^n \tau_{o} (v_k)$. Here we remark that if $v_{k-1} = v_{k+1}$ for some $k$, then 
depending on the orientation of $\Gamma$, either all the faces incident to $v_k$ should be considered lying on the right of the path or none of them should be 
considered on the right. However, we will not study the latter case throughout the paper, so in this case all the faces incident to $v_k$ should be 
considered on the right, hence we will have $s_k = \deg v_k$ and  $\tau_{o} (v_k) + \kappa (v_k) = 1/2$.
Finally if $\Gamma$ is a cycle of length zero, that is, if $\Gamma = [v_0]$ for some $v_0 \in V$,  we define $\tau_{o} (v_0) = \tau_{o} (\Gamma) = 1 - \kappa(v_0)$.

 \begin{figure}[t]
\centerline{
\hfill \subfigure[the case $v_{k-1} \ne v_{k+1}$]{\scalebox{.95}{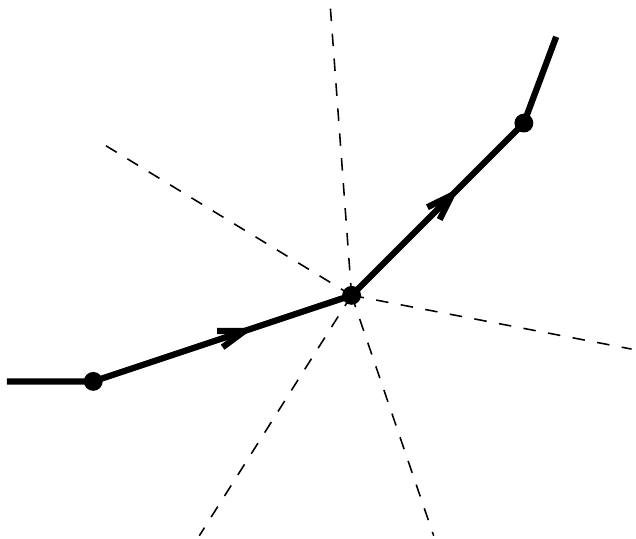}} \hfill
\subfigure[the case $v_{k-1} = v_{k+1}$]{\scalebox{.95}{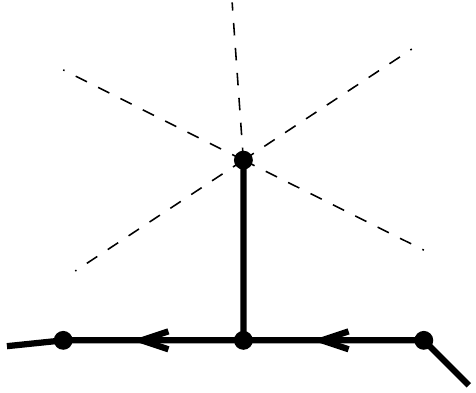}} \hfill
}
\caption{Faces incident to $v_k$: the faces $f_j^{(k)}$ are on the right of $[v_{k-1}, v_k, v_{k+1}]$, and the faces $g_j^{(k)}$ are on the left of $[v_{k-1}, v_k, v_{k+1}]$.  
We do not consider the case (b) for inner left turns. That is, a path like  (b)  with opposite orientation will not appear throughout the paper.}\label{F:leftturn}
\end{figure}

Perhaps it would be better to use the notation like $\tau_o(v_k; v_{k-1}, v_{k+1})$ instead of $\tau_{o} (v_k)$, because the quantity in \eqref{outerLT}
 is not a function for $v_k \in V$ but a function for the triple $[v_{k-1}, v_k, v_{k+1}]$. However,  we will stick to the notation $\tau_o (v_k)$ for simplicity, as long as it does not cause confusion.
The same rule will apply to the next object,  \emph{inner left turn} of a cycle. For this,
we assume that the cycle $\Gamma = [v_0, \ldots, v_n= v_0]$ is of length $n \geq 3$ and $v_{k-1} \ne v_{k+1}$ for all $k =1,2, \ldots, n$, where we interpret $v_{n+1} = v_1$ as before.
For each $k$, let $g_1^{(k)}, g_2^{(k)}, \ldots, g_{t_k}^{(k)}$ be the faces of $G$ incident to $v_k$
and \emph{lying on the left} of the path $[v_{k-1}, v_k, v_{k+1}]$. Then the \emph{inner left turn} occurred at $v_k$ is
\begin{equation}\label{innerLT}
\tau_{i} (v_k) =   \frac{1}{2} - \sum_{j=1}^{t_k} \left( \frac{1}{2} - \frac{1}{\deg g_j^{(k)}} \right),
\end{equation}
and the inner left turn of $\Gamma$ is the sum $\tau_{i} (\Gamma) = \sum_{k=1}^n \tau_{i} (v_k)$. Note that if $\Gamma$ is simple, we have
$\tau_{i} (v_k) - \tau_{o} (v_k) = \kappa (v_k)$ for all $k$.

The meaning of outer  left turn is the following. Suppose a cycle $\Gamma = [v_0, \ldots, v_n]$ is given on the polyhedral surface $\mathcal{S}_G$.
Imagine that  a person stands  one step to the right  from $\Gamma$,  and walks side by side along $\Gamma$.
Then  $2 \pi \cdot \tau_{o} (v_k)$ would be the angle by which he or she turns to the left near $v_k$, and 
the total left turn made after a complete rotation along $\Gamma$ would be $2 \pi \cdot \tau_{o} (\Gamma)$. For the inner left turn, we think that this person stands one step to the left from $\Gamma$
and walks, and observe that $2 \pi \cdot \tau_{i} (v_k)$ is the left turn made near $v_k$.  Thus $2 \pi \cdot \tau_{i} (\Gamma)$ will beome the total left turn made after a complete rotation along $\Gamma$.

\begin{figure}[t]
\begin{center}
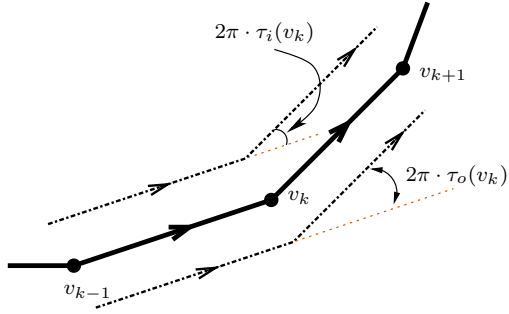
\caption{Meaning of outer left turn and inner left turn.}\label{F:leftturn-meaning}
\end{center}
\end{figure}

We next suppose a connected finite subgraph $S \subset G$ is given, and let us assume that  $\mathbb{C} \setminus D(S)$ has $m$ components, where $D(S)$ is as in \eqref{D(S)}. 
Here $S$ does not have to be induced, but as we mentioned at the end of the previous section we can write
$bS = \Gamma_1 \cup \Gamma_2 \cup \cdots \cup \Gamma_m$,
where each $\Gamma_j$ is a cycle corresponding to the topological boundary of a component of $\mathbb{C} \setminus D(S)$.
Now we define the \emph{outer left turn}  of $b S$ by $\tau_{o} (b S) = \sum_{j=1}^m \tau_{o} (\Gamma_j)$ and obtain the following theorem.

\begin{theorem}[Combinatorial Gauss-Bonnet Theorem-Type I]\label{CGBT1}
Suppose $G$ is an infinite tessellation and $S \subset G$ a connected finite subgraph of $G$. Then we have
\begin{equation}\label{GBF-1}
\kappa(S) + \tau_{o} (b S) = \chi (S),
\end{equation}
where  $\chi (S) = |V(S)| - |E(S)| + |F(S)|$, the Euler characteristic of $S$, and $\kappa (S) = \sum_{v \in V(S)} \kappa (v)$.
\end{theorem}
\begin{proof}
For a sufficiently small $\epsilon >0$, let $P$ be the polygonal region in $\mathcal{S}_G$ which is obtained from the $\epsilon$-neighborhood of $D(S) \subset \mathcal{S}_G$ by \emph{sharpening corners} (Figure~\ref{F:sharpening}). 
\begin{figure}[t]
\begin{center}
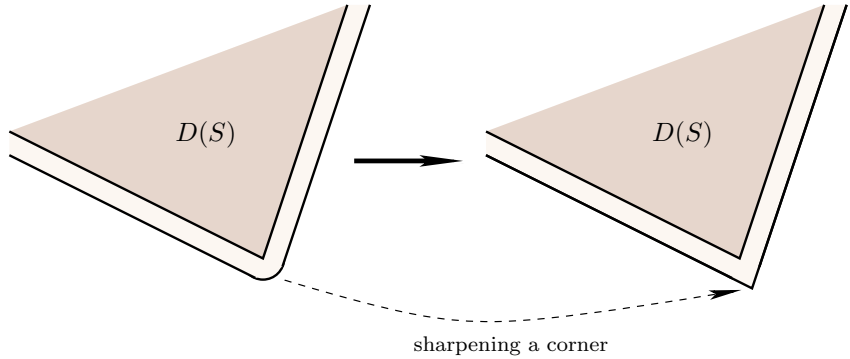
\caption{Obtaining the polygonal region $P$ by sharpening corners. This operation does not change geodesic curvature of the boundary.}\label{F:sharpening}
\end{center}
\end{figure}
Let $b P$ be the topological boundary of $P$ and $\tau (b P)$ be the total left turn, or the geodesic curvature, of $b P$. Then one can easily check that
\[
\tau (b P) = 2 \pi \cdot \tau_{o} (b S).
\]

Note that the polyhedral surface $\mathcal{S}_G$ is locally Euclidean except the vertices of $G$. Thus we have $\omega (P^\circ) = 2 \pi \cdot \kappa (S)$ by \eqref{curvertex}, 
where   $P^\circ$  is the interior of $P$. Now 
Theorem~\ref{CGBT1} follows from the Gauss-Bonnet Theorem for polyhedral surfaces (cf. \cite[p.\ 214]{AZ})
\begin{equation}\label{GBT-S}
\omega (P^\circ) + \tau (b P) = 2 \pi \cdot \chi(P),
\end{equation}
because we have $\chi (P) = \chi (S)$.
\end{proof}

We will call \eqref{GBF-1}  the first combinatorial Gauss-Bonnet formula involving boundary turns, or GBF-1 for abbreviation. To deduce another combinatorial Gauss-Bonnet formula we need,
let a  finite subgraph $S \subset G$ be given, which is not necessarily connected. We remove from $S$ all the vertices and edges not incident to faces in $F(S)$, and obtain another subgraph $S_0$.
Note that $S_0$ consists of faces of $S$; i.e., we have $D(S_0) = \bigcup_{f \in F(S)} f$. Now in $\mathcal{S}_G$, we assume that $D(S_0)^\circ$, the interior of $D(S_0)$, has $m$ components.
Then we can write $b S_0 = \Gamma_1 \cup \cdots \cup \Gamma_m$, where each $\Gamma_j$ corresponds to the (topological) boundary of  a component of $D(S_0)^\circ$.
Thus each $\Gamma_j$ must be a union of cycles, so we can write $\Gamma_j = \Gamma_j^1 \cup \Gamma_j^2 \cup \cdots \cup \Gamma_j^{l_j}$, where each $\Gamma_j^k$ is a cycle
corresponding to a component of complements of the component of $D(S_0)^\circ$.

For example let $S$ be the graph in Figure~\ref{F:D(S)}. From $S$ we remove the vertices $v_6, v_7, v_8$ and edges $[v_5, v_6], [v_6, v_7], [v_6, v_8], [v_9, v_{10}]$ 
to obtain $S_0$. Then the region $D(S_0)^\circ$ has two components, and we can write $b S_0 = \Gamma_1 \cup \Gamma_2= (\Gamma_1^1 \cup \Gamma_1^2) \cup \Gamma_2$, 
where $\Gamma_1^1 = [v_1, v_2, v_3, v_4, u_1, u_2, u_3, v_4, v_5, v_9, v_1]$, $\Gamma_1^2 = [w_1, w_2, w_3, w_1]$, and $\Gamma_2 = [v_{10}, v_{11}, v_{12}, v_{10}]$.

Note that  every $\Gamma_j^k$ should be of length at least $3$ and every vertex in $\Gamma_j^k$ has distinct front and back vertices, 
hence  $\tau_i (\Gamma_j^k)$  is well-defined for all $k$. Thus $\tau_i (\Gamma_j) = \sum_{k=1}^{l_j} \tau_i (\Gamma_j^k)$ is also well-defined. 
Let $b_i S = b S_0$, which we call  the \emph{inner boundary walk} of $S$,  and define the \emph{inner left turn} of $b_i S$ by the sum $\tau_i (b_i S) = \sum_{j =1}^m \tau_i (\Gamma_j)$.
(To distinguish the boundary walk $b S$ from the inner boundary walk $b_i S$, we will call $b S$ the \emph{outer} boundary walk of $S$ whenever needed.)
In the proof of Theorem~\ref{CGBT2} below we will \emph{shrink} $D(S)$ slightly, so that the inner left turn $\tau_i (b_i S)$ 
becomes $2 \pi$ times the geodesic curvature of the boundary of the shrunken domain.
Finally let $S_{-} = V(S) \setminus V(b S)$, the set of \emph{interior vertices} of $S$, and $\chi (S_{-}) = \chi (D(S)^\circ)$, the Euler characteristic of the topological interior of $D(S)$.
Now we finish this section with the following form of the Gauss-Bonnet formula.

\begin{theorem}[Combinatorial Gauss-Bonnet Theorem-Type II]\label{CGBT2}
Suppose $G$ is an infinite tessellation and $S \subset G$ a finite subgraph of $G$. Then we have
\begin{equation}\label{GBF-2}
\kappa(S_{-}) + \tau_{i} (b_i S) = \chi (S_{-}),
\end{equation}
where $\kappa (S_{-}) = \sum_{v \in S_{-}} \kappa (v)$.
\end{theorem}
\begin{proof}
If $b S$ is  just a simple cycle $[v_0, \ldots, v_n]$, then equation \eqref{GBF-2} can be deduced from GBF-1 and the fact $\tau_{i} (v_k) - \tau_{o} (v_k) = \kappa (v_k)$.
In general we can obtain \eqref{GBF-2} by shrinking $D(S)$ slightly. Namely, for sufficiently small $\epsilon>0$, let $(bS)_\epsilon$ be the $\epsilon$-neighborhood of $b S \subset \mathcal{S}_G$
and let $D(S)_\epsilon = D(S) \setminus (b S)_\epsilon$. Then by sharpening corners of $D(S)_ \epsilon$ we obtain a polygonal region $P_{-}$ such that
$\tau (b P_{-}) = 2 \pi \cdot \tau_i (b_i S)$. Now the Gauss-Bonnet formula \eqref{GBT-S} for polyhedral surfaces implies \eqref{GBF-2} by a simple computation as in Theorem~\ref{CGBT1}.
\end{proof} 

The formula \eqref{GBF-2} will be called the second combinatorial Gauss-Bonnet formula involving boundary turns, or GBF-2 for abbreviation.  

\section{Proof of Theorem~\ref{T1}}\label{S:T1}

In this section we will prove Theorem~\ref{T1}. Let $G=(V, E, F)$ be a disk triangulation graph, and $k_n$ and $a_n$ be the sequences defined in Theorem~\ref{T1}
for a fixed vertex $v_0 \in V$ and combinatorial balls $B_n = B_n (v_0)$. Also let $S_n = S_n (v_0)$ be the combinatorial spheres.

Now we fix $n \geq 1$, and present a series of observations below.
 
\begin{observation}\label{O1}
Every vertex $v$ in $b B_n$ is incident to an edge going outside. That is, if $v$ is a vertex such that a sequence of the form $[\cdots, u, v, w, \cdots]  \subset b B_n$
corresponds to the boundary of a component of $\mathbb{C} \setminus D(B_n)$, 
then there exists at least one edge in $\partial B_n \subset E \setminus E(B_n)$  incident to $v$ and lying on the right of $[u, v, w]$, 
or equivalently, there are at least two faces in $F \setminus F(B_n)$ that are incident to $v$ and lie  on the right of $[u, v, w]$. 
\end{observation}

If Observation~\ref{O1} is not true,  there should be only one face, say $f$, lying on the right of $[u, v, w]$. Then  $f$ must be a triangle with vertices $u, v$, and $w$,
because $G$ is a disk triangulation graph. Thus $f \in F(B_n)$ by the definition of the face set $F(B_n)$, because we assumed that $B_n$ was an \emph{induced} subgraph. This implies
that the sequence $[\ldots, u, v, w, \ldots]$ cannot correspond to  a component of $\mathbb{C} \setminus D(B_n)$, 
since  in this case we should have had $[\ldots, u, w, \ldots]$ instead.  This contradiction confirms Observation~\ref{O1}. Also note that
such $u$ and  $w$ do not have to be unique for a fixed $v \in V(b B_n)$, but Observation~\ref{O1} holds for every sequence of the form $[\ldots, u, v, w, \ldots]$ shown in $b B_n$
that corresponds to the boundary of a component of $\mathbb{C} \setminus D(B_n)$.

\begin{observation}\label{O2}
Every vertex $v$ in $b B_n$ belongs to $S_n$; i.e., the combinatorial distance between $v_0$ and $v$ must be $n$.
\end{observation}

Observation~\ref{O2} is trivial since every $v \in V(b B_n)$ is incident to an edge in $\partial B_n$ by Observation~\ref{O1}, hence $v$ has a neighbor in $V \setminus B_n$.

\begin{observation}\label{O5}
For any set $K \subset V(b B_n)$, vertices in $B_n \setminus K$ can be connected  to each others by paths lying in $B_n \setminus K$.
\end{observation}

Observation~\ref{O5} is also trivial, because $K$ is a subset of  $S_n$ by Observation~\ref{O2} and every vertex in $B_n$ can be connected to $v_0$ by 
a shortest path lying in $B_n$, which cannot pass through vertices in $S_n$ except the ends.

\medskip

We next want to remove some edges and faces from $B_n$ in order to avoid  pathological situations. But because the face set of a subgraph $S \subset G$ is determined
by its vertex and edge sets as we saw in Section~\ref{prelim}, we will only discuss how to remove edges from $B_n$.
Let $D_0$ be the component of $D(B_n)^\circ$ containing $v_0$, where $(\cdot)^\circ$ denotes the topological interior of the given set. Define $A_n$ as the subgraph of $B_n$
consisting of vertices and edges in $\overline{D}_0$, the topological closure of $D_0$. Then it is not difficult to see that 
$D(A_n) = \overline{D}_0$, which we call the \emph{main body} of $D(B_n)$.

Now suppose $\Gamma$ is a simple cycle in $b B_n$, which might correspond to more than one component of $\mathbb{C} \setminus D(B_n)$, such that $D(B_n) \setminus \Gamma$
is disconnected. Then except the component of $D(B_n) \setminus \Gamma$  containing $v_0$,  the other
components of $D(B_n) \setminus \Gamma$ cannot contain vertices since otherwise $B_n \setminus V(\Gamma)$ would have more than one component containing vertices,
contradicting Observation~\ref{O5}. 
It is also clear from the definition of $A_n$ that each edge in $E(A_n)$ either lies on $\Gamma$ or belongs to the component of  $D(B_n) \setminus \Gamma$  containing $v_0$.
That is, except the component including the main body,  every component of $D(B_n) \setminus \Gamma$ is composed of edges and faces in $D(B_n) \setminus D(A_n)$
with vertices deleted (and perhaps some edges deleted, depending on $\Gamma$). Conversely, edges in $E(B_n) \setminus E(A_n)$ can definitely be separated 
by simple cycles in $b A_n \subset b B_n$ from the main body. Therefore  $A_n$ is the subgraph of  $B_n$ which can be obtained from $B_n$ by 
deleting all the edges separated from the main body by simple cycles in $b B_n$. See Figure~\ref{F:betterball}.
\begin{figure}[t]
\begin{center}
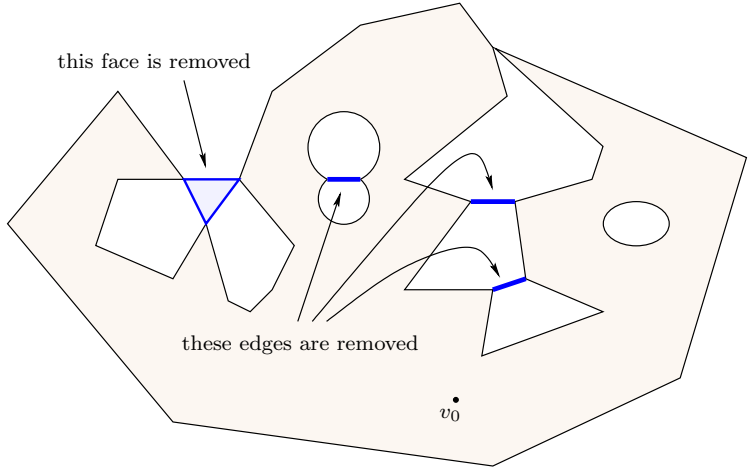
\caption{Removing edges and faces lying in the outside of the main body.}\label{F:betterball}
\end{center}
\end{figure}

It is not difficult to see that $V(b A_n) = V(b B_n)$, because we did not remove vertices from $B_n$ and every vertex in $b B_n$ has a neighbor in $V \setminus V(B_n)$
by Observation~\ref{O1}.  Therefore Observations~\ref{O1} and \ref{O2} remain true with $A_n$ in  place of $B_n$.
Moreover, because $V(A_n) \subset S_n$ and  edges in $E(B_n) \setminus E(A_n)$ must have both ends on $S_n$, 
the arguments in Observation~\ref{O5} work for $A_n$ as well; i.e., Observation~\ref{O5} also holds with $A_n$ in  place of $B_n$.

Suppose that $\mathbb{C} \setminus D(A_n)$ has $m$ components. Then as discussed at the end of Section~\ref{prelim} it is possible to write
$b A_n = \Gamma_1 \cup \cdots \cup \Gamma_m$, where each $\Gamma_j$ is a cycle corresponding to a component of $\mathbb{C} \setminus D(A_n)$. 

\begin{observation}\label{O3}
Every $\Gamma_j$ is a \emph{simple} cycle. That is, if $\Gamma_j = [v_1, \ldots, v_{l} = v_1]$ for some $j$, then $v_a \ne v_b$ for $1 \le a < b \le l-1$.
\end{observation}

If not, there exists $v \in b A_n$ that repeats in the cycle  $[v_1, v_2, \ldots, v_{l} = v_1]$. Then $A_n \setminus \{v\}$ is disconnected,
because we chose each $\Gamma_j$ as the boundary of a component of $\mathbb{C} \setminus D(A_n)$ (see (a) and (b) of Figure~\ref{F:cutvertex}).
This contradicts Observation~\ref{O5}, proving Observation~\ref{O3}.

 \begin{figure}[t]
\centerline{
\hfill \subfigure[]{\scalebox{1}{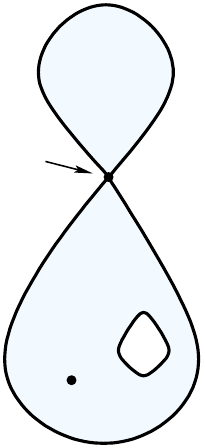}} \hfill
\subfigure[]{\scalebox{1}{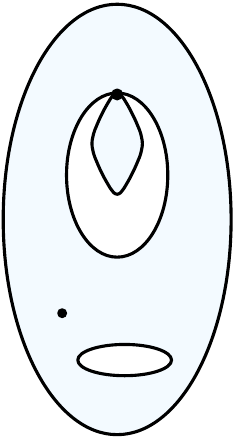}} \hfill
\subfigure[]{\scalebox{1}{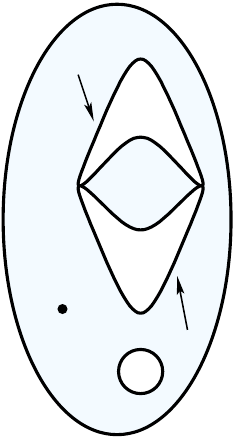}} \hfill
}
\caption{Cases that $A_n \setminus K$ is disconnected for some $K \subset V(b A_n)$.}\label{F:cutvertex}
\end{figure}

\begin{observation}\label{O4}
For $i  \ne j$, $\Gamma_i \cap \Gamma_j$ is at most a vertex.
\end{observation}

If $\Gamma_i \cap \Gamma_j$  contains an edge, then this edge should have been removed from $B_n$ while obtaining $A_n$, 
because on this edge we see faces in $F \setminus F(A_n)$ on both sides. This  contradiction shows that every component of $\Gamma_i \cap \Gamma_j$  is at most a vertex.
Now if $\Gamma_i \cap \Gamma_j$ has more than one component, then   $A_n \setminus \Gamma_i \cap \Gamma_j$ has more than one component and
we have another contradiction. See Figure~\ref{F:cutvertex}(c).  We conclude that  
if $\Gamma_i \cap \Gamma_j$ is nonempty, then it must be a vertex.

\medskip

By extending  the previous argument, we can  prove the following.

\begin{observation}\label{O6}
If $\Gamma_{i_1} \cup \Gamma_{i_2} \cup \cdots \cup \Gamma_{i_k}$ is connected and $j \notin \{ i_1, i_2, \ldots, i_k \}$,
then $(\Gamma_{i_1} \cup \Gamma_{i_2} \cup \cdots \cup \Gamma_{i_k}) \cap \Gamma_j$ is at most a vertex.
\end{observation}
 
By the same argument as the above we see that every component of $(\Gamma_{i_1} \cup  \cdots \cup \Gamma_{i_k}) \cap \Gamma_j$ is at most a vertex.
Now if $(\Gamma_{i_1}  \cup \cdots \cup \Gamma_{i_k}) \cap \Gamma_j$ has more than one component, then definitely the set
$\Gamma_j \setminus (\Gamma_{i_1}  \cup \cdots \cup \Gamma_{i_k})$ also has more than one component.
Let $\Gamma_j^1$ be one of such component. Then it is not difficult to see that $(\Gamma_{i_1}  \cup \cdots \cup \Gamma_{i_k})\cup \Gamma_j^1$ separates 
$D(A_n)$ into more than one component, contradicting the way we obtained $A_n$; i.e., in this case there is a simple cycle 
$\Gamma \subset (\Gamma_{i_1}  \cup \cdots \cup \Gamma_{i_k})\cup \Gamma_j^1$ that separates 
$D(A_n)$ into more than one component. This contradiction proves Observation~\ref{O6}. (Note that if $\Gamma_j^1$ is an edge, 
then $(\Gamma_{i_1}  \cup \cdots \cup \Gamma_{i_k})\cup \Gamma_j^1$ could separate a face (with  vertices removed) from the main body.
Look at the removed face in Figure~\ref{F:betterball}.  
Thus the argument above yields a result contradictory to the way we obtained $A_n$, but it does not contradict  Observation~\ref{O5}.  
This is partially why we consider $A_n$ instead of $B_n$.)

\medskip

Define $| b A_n| = \sum_{j=1}^m \mbox{length}(\Gamma_j)$, and we next prove the following lemma.

\begin{lemma}\label{L:S_n}
For each $n \geq 1$, we have the inequality
\begin{equation}\label{E:S_n} 
|S_n| \geq |b A_n| - (m-1),
\end{equation}
where $|S_n|$ denotes the number of vertices in $S_n$, and $m$ is the number of components of $\mathbb{C} \setminus D(A_n)$.
\end{lemma}

\begin{proof}
Because $V(b A_n) \subset S_n$ by Observation~\ref{O2} and $|V(\Gamma_j)| = \mbox{length}(\Gamma_j)$ for every $j$ by Observation~\ref{O3}, it suffices to show that 
\[
|V(b A_n)| = | V( \Gamma_1 \cup \Gamma_2 \cup \cdots \cup \Gamma_m)| \geq  \sum_{j=1}^m |V(\Gamma_j)| - (m-1),
\]
or
\begin{equation}\label{seqpf}
 | V( \Gamma_1 \cup \Gamma_2 \cup \cdots \cup \Gamma_j)| \geq |  V( \Gamma_1 \cup \Gamma_2 \cup \cdots \cup \Gamma_{j-1})|+|V(\Gamma_j)|-1
 \end{equation}
 for every $j=2,3, \ldots, m$. In fact, \eqref{seqpf} does not have to be true in general, but it will be true by renumbering $\Gamma_j$'s appropriately.
 
 If $\Gamma_i \cap \Gamma_j = \emptyset$ for every $i$ and $j$  with $i \ne j$, then \eqref{seqpf} is definitely true and the proof has been completed. If not, 
 then we may assume that  $\Gamma_1 \cap \Gamma_2 \ne \emptyset$ after renumbering $\Gamma_j$'s if necessary, and we see that $\Gamma_1 \cap \Gamma_2$
 is a vertex by Observation~\ref{O4}. Thus \eqref{seqpf} holds in this case. Now if there is a cycle, say $\Gamma_3$ after renumbering, such that 
 $(\Gamma_1 \cup \Gamma_2) \cap \Gamma_3 \ne \emptyset$, then we know that the intersection is just a vertex by Observation~\ref{O6}, hence \eqref{seqpf} still holds for $j=3$.
 
 We keep doing the above process until there exists $l_1$, $1 < l_1 \leq m$, such that \eqref{seqpf} holds for all $j=2, \ldots, l_1$, the set 
 $\Gamma_1 \cup \ldots \cup \Gamma_{l_1}$ is connected, and  $(\Gamma_1 \cup \cdots \cup \Gamma_{l_1})\cap \Gamma_k  = \emptyset$ for every $k > l_1$. 
 Then we choose $\Gamma_{l_1 +1}$ as it is indexed, and we see that 
 \[
 | V( \Gamma_1 \cup \Gamma_2 \cup \cdots \cup \Gamma_{l_1} \cup \Gamma_{l_1 +1})| = |  V( \Gamma_1 \cup \Gamma_2 \cup \cdots \cup \Gamma_{l_1})|+|V(\Gamma_{l_1 +1})|,
\]
so \eqref{seqpf} holds for $j=l_1 +1$. If there is $k > l_1 +1$ such that $(\Gamma_1 \cup \cdots \cup \Gamma_{l_1+1})\cap \Gamma_k  \ne \emptyset$, then definitely we have 
$(\Gamma_1 \cup \cdots \cup \Gamma_{l_1})\cap \Gamma_k  = \emptyset$  and $\Gamma_{l_1 +1 } \cap \Gamma_k \ne \emptyset$, hence Observation~\ref{O6} is still applicable and
\eqref{seqpf} holds for $j=l_1 +2$ after renumbering if necessary. We keep doing this process until there exists $l_2$, $l_1 < l_2 \leq m$,
 such that \eqref{seqpf} holds for all $j=2, \ldots, l_2$, the set 
 $\Gamma_{l_1+1} \cup \ldots \cup \Gamma_{l_2}$ is connected, and  $(\Gamma_1 \cup \cdots \cup \Gamma_{l_2})\cap \Gamma_k  = \emptyset$ for every $k > l_2$. 
 Then we choose $\Gamma_{l_2 +1}$ as it is indexed, and we can go to the next step.
 
 The above process can be extended to the $(m-1)$th step, hence \eqref{seqpf} holds for all $j=2,3, \ldots, m$ with an appropriate enumeration of $\Gamma_j$'s. This completes the proof.
\end{proof}

Suppose $v$ is a vertex that appears in $\Gamma_j$; i.e., we assume that $\Gamma_j$ is written as $\Gamma_j = [ \ldots, u, v, w, \ldots ]$ for some vertices $u, w \in V(b A_n) \subset S_n$.
Then as we saw in Observation~\ref{O1} there is at least one edge incident to $v$ and going outside, or on the right of $[u,v,w]$. Let $e_1, e_2, \ldots, e_k$, $k \geq 1$, be the edges
incident to $v$ and lying on the right of $[u,v,w]$, which we enumerate counterclockwise. Then we say that there are $k-2$ \emph{extra edges} at $v$ along $\Gamma_j$, and if $k \geq 3$ the
edges $e_2, \ldots, e_{k-1}$ will be called extra edges. Thus if $k=1$,  there is `$-1$' extra edge at $v$ (along $\Gamma_j$) in our definition.
On the other hand, clearly every vertex in $S_{n+1}$ is connected  to $b B_n$ by an edge. We can say more: a vertex in $S_{n+1}$  
must be connected to  $b A_n$  because we did not remove vertices while obtaining $A_n$. 
Thus we can classify vertices in $S_{n+1}$ as follows: a vertex in $S_{n+1}$ is called  type~I if it is incident to a face with one edge on $b A_n$, and it is called  type~II
if it is connected to $b A_n$ by an extra edge. There could be some vertices in $S_{n+1}$ which are of both types~I and II simultaneously   (Figure~\ref{F:S_n}).
\begin{figure}[t]
\begin{center}
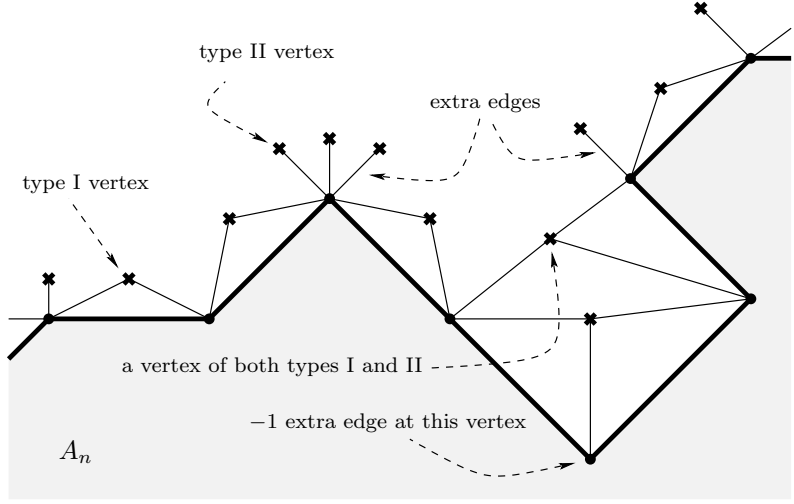
\caption{Vertices in $S_n$ and $S_{n+1}$,  and edges between them.}\label{F:S_n}
\end{center}
\end{figure}

Each  type I vertex corresponds to at least one edge in $b A_n$, and if there is a vertex on some $\Gamma_j$ with $-1$ extra edge then
the number of  type I vertices would be reduced, say from $|b A_n|$, by 1. Furthermore every type II vertex  corresponds to at least one extra edge, so we conclude that
\begin{equation}\label{Sn+1}
|S_{n+1}| \leq   |b A_n|  + (\mbox{total number of extra  edges}),
\end{equation}
where the `total number of extra edges' means the sum of the number of extra edges at every vertex along every cycle $\Gamma_j$.
On the other hand, if there are $k$ extra edges at $v$ along $\Gamma_j$, then because $G$ is a disk triangulation graph and the number of faces incident to $v$ and lying on
the right of $\Gamma_j$ is $k+3$, 
 the outer left turn $\tau_o(v)$  occurred at $v$ (along $\Gamma_j$), which is defined in \eqref{outerLT},  is
\[
\tau_o(v) = \left( \frac{1}{2} - \frac{1}{3} \right) (k+3) - \frac{1}{2} = \frac{k}{6}.
\]
Therefore we have
\[
\tau_o (b A_n) = \frac{1}{6} \cdot  (\mbox{total number of extra  edges}).
\]

For each $v \in V$, the combinatorial curvature at $v$ becomes 
\[
\kappa(v) = 1 - \frac{\deg v}{2} + \sum_{f \sim v} \frac{1}{\deg f} = 1 - \frac{\deg v}{6}
\]
because $G$ is a disk triangulation graph. Thus 
\[
\kappa (A_n) = \kappa(B_n) = \sum_{v \in B_n} \kappa(v) = \frac{1}{6} \cdot \sum_{v \in B_n} (6 - \deg v),
\]
hence  GBF-1 \eqref{GBF-1} implies that
\begin{gather*}
\frac{1}{6}  \cdot \sum_{v \in B_n} (6 - \deg v) + \frac{1}{6} \cdot  (\mbox{total number of extra  edges})\\
 = \kappa (A_n) + \tau_o (b A_n) = \chi (A_n) = 2 - m,
\end{gather*}
or
\begin{equation}\label{degreex}
\sum_{v \in B_n} (\deg v - 6) = 6m-12 + (\mbox{total number of extra  edges}).
\end{equation}
Note that \eqref{degreex} is true even for $n=0$, since $B_0 = \{ v_0 \}$ is simply connected and every edge incident to $v_0$ is an extra edge.
Now by \eqref{E:S_n},  \eqref{Sn+1}, and \eqref{degreex} we get
\begin{align*}
|S_{n+1}| - |S_n|  & \leq  \bigl( |b A_n|  + (\mbox{total number of extra  edges})\bigr)  - \bigl( |b A_n| - (m-1) \bigr) \\
                             & = (\mbox{total number of extra  edges}) +m -1 \\
                             & = \sum_{v \in B_n} (\deg v - 6)  -5m +11 \leq   \sum_{v \in B_n} (\deg v - 6) +6 
\end{align*}
because $m \geq 1$. Therefore, because $k_n = \sum_{v \in B_n} (\deg v - 6)$  and $a_n = \sum_{j=0}^{n-1} (k_j +6)$, we have
\begin{equation}\label{spheresize}
\begin{aligned}
|S_{n+1}| & = |S_1|+ \sum_{j=1}^n (|S_{j+1}| - |S_j|) \\
& \leq \deg v_0 + \sum_{j=1}^n \left( \sum_{v \in B_j} (\deg v - 6) +6 \right) \\
& = \{ (\deg v_0 -6) +6 \} + \sum_{j=1}^n (k_j+6)  \\
& =  (k_0 +6) + \sum_{j=1}^n (k_j + 6) =  \sum_{j=0}^n (k_j + 6) =  a_{n+1}
\end{aligned}
\end{equation}
and we see that  $\sum_{n=1}^\infty 1/|S_n| = \infty$ if \eqref{Reci} holds; i.e., if we have
\begin{equation}\label{reci3}
 \sum_{n=1}^\infty \frac{1}{a_n} = \infty.
 \end{equation} 
Therefore by Rodin and Sullivan's theorem(Theorem~\ref{T:RoSull}) we conclude that $G$ is CP parabolic if \eqref{reci3} holds.

To prove the recurrence part of Theorem~\ref{T1}, let $E_{n} = \partial B_n$, the set of edges connecting $B_n$ and $S_{n+1}$.
Then because $|E_n| \leq 2 |b A_n| + (\mbox{total number of extra  edges})$, we have from \eqref{E:S_n},  \eqref{degreex}, and \eqref{spheresize} that
\begin{equation}\label{edgesize}
\begin{aligned}
|E_n| & \leq 2 |b A_n| +  (\mbox{total number of extra  edges}) \\
          &  \leq 2 (|S_n| + m-1) + \left( \sum_{v \in B_n} (\deg v - 6) - 6m+12 \right) \\
          & = 2 \cdot |S_n| +k_n  -4m + 10 \leq 2  \cdot\sum_{j=0}^{n-1} (k_j +6)  + k_n +6  \\
          & =  \sum_{j=0}^{n-1} (k_j +6) + \sum_{j=0}^{n} (k_j +6) = a_n + a_{n+1}.
\end{aligned}  
\end{equation}           
Therefore if \eqref{Reci2} holds, that is, if
\[
\sum_{n=1}^\infty \frac{1}{a_n + a_{n+1} }= \infty,
\]
then  we  have  $\sum_{n=1}^\infty 1/|E_n| = \infty$ and recurrence of $G$ follows from  Nash-Williams' theorem(Theorem~\ref{T:Nash}).
This completes the proof of Theorem~\ref{T1}, and a remark follows.

\begin{remark}\label{rm1}
By \eqref{spheresize} we have $|S_n| \leq a_n$  for all $n=1,2, \ldots$. 
(In fact, we proved it only for $n \geq 2$ since we have assumed $n \geq 1$ in the proof of \eqref{spheresize}, but one can easily check that $|S_1| = \deg v_0 =a_1$). 
Then because $G$ being a disk triangulation graph forces $|S_n|$ to be  greater than or equal to $3$ for every $n \geq 1$ (because $G$ has to be \emph{simple} and \emph{infinite}), 
we must have $a_n \geq 3$, $n \in \mathbb{N}$, as in  \eqref{ICP} of Theorem~\ref{T:S}. 
Though very far from being hyperbolic, the disk triangulation graph satisfying $|S_n| =3$ for $n \geq 1$ has the degree excess sequence $k_n$ such that $k_0 = -3$ and $k_n = -6$ for $n \geq 1$, 
hence the equality $a_n = \sum_{j=0}^{n-1} (k_j + 6)=3$ holds in this case (Figure~\ref{extreme3}). 
\end{remark}

 \begin{figure}[t]
\begin{center}
 \includegraphics{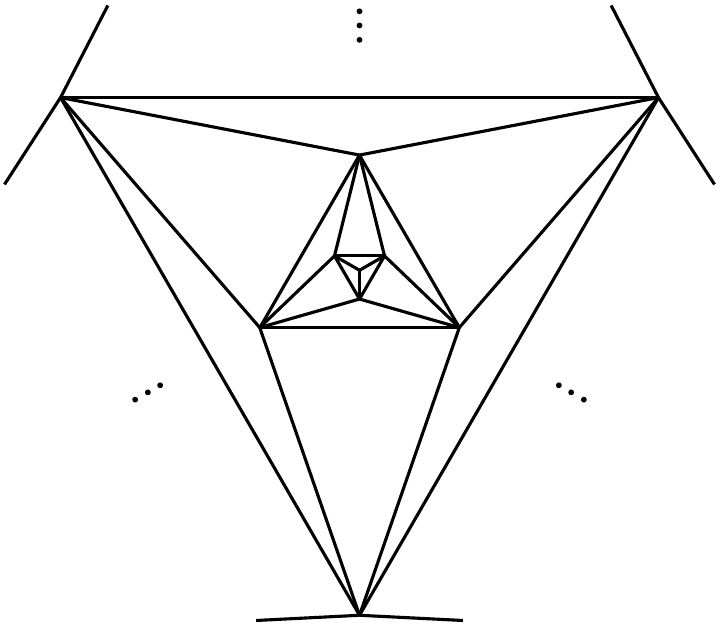}
 \caption{The disk triangulation graph with $|S_n|=3$ for $n \geq 1$.}\label{extreme3}
\end{center}
\end{figure}

\section{Proof of Theorem~\ref{T:Hyp}}\label{S:T2}
If the assumption of Theorem~\ref{T:Hyp} holds for a disk triangulation graph $G$, then there exist a positive number $c >0$ 
and  a nonempty connected  finite subgraph $S_0 \subset G$ such that
\begin{equation}\label{suff1}
\sum_{v \in S} (\deg v - 6) \geq c \cdot g(|S|)
\end{equation}
for every  connected  finite subgraph $S \supset S_0$. Here $g$ is a nondecreasing function such that $\sum_{n=1}^\infty 1/g(n)^2 < \infty$ as in the statement of Theorem~\ref{T:Hyp}.
Thus to prove Theorem~\ref{T:Hyp},  we  assume that a connected  finite subgraph $S$  such that $S_0 \subset S \subset G$ is given. 

Let $T$ be the induced subgraph of $G$ satisfying $V(T) = V(S) \cup d S$, and we add to $T$ all the vertices
and edges in the bounded components of $\mathbb{C} \setminus D(T)$ and obtain a simply connected graph $W$.
That is, we fill all the `holes' in $D(T)$ to obtain $W$.

We claim that for any  $K \subset V(bW)$, the induced graph with the vertex set $V(W) \setminus K$ is connected as in Observation~\ref{O5}. 
To see this, first note that $b W \subset b T$
corresponds to the unbounded component of $\mathbb{C} \setminus D(T)$. Moreover, every vertex in the boundary walk $b T$ 
belongs to $d S$, the vertex boundary of $S$, because if $w \in V(S)$ then all the neighbors of $w$ 
must be in $V(T)$, hence $w$ must be in the interior of $D(T)$ and $w \notin b T$. Now suppose that $v \in V(W) \setminus K$. 
Then it belongs to one of $V(S)$, $dS$, or a bounded component of $\mathbb{C} \setminus D(T)$. 
If $v$ is either in $V(S)$ or $dS$, then it can be connected by a path lying in  $W \setminus K$ to any other vertices in $V(S)$, because $S$ is connected and $S \cap K = \emptyset$.
(Note that since $K \subset V(bW) \subset V(bT) \subset d S$, we have $S \cap K = \emptyset.$) 
Suppose $v$ lies in a bounded component of $\mathbb{C} \setminus D(T)$, and let us write $b T$  as the union of cycles  
corresponding to the boundaries of components of  $\mathbb{C} \setminus D(T)$, as we did before. Note that one of such cycles must be $bW$, and let 
$\Gamma$ be the cycle corresponding to the component containing $v$. Then as in Observation~\ref{O3}, both $bW$ and $\Gamma$ are simple cycles,
since otherwise $\Gamma$ minus the repeated vertex (or $bW$ minus the repeated vertex) must contain a vertex, say $w$, which cannot be connected to
the \emph{connected} graph $S$ without passing through a vertex in $V(\Gamma) \setminus \{ w \}$ (respectively $V(bW) \setminus \{w\}$),
while every vertex in $V(bT) \subset dS$ can be connected to $S$ by an edge. 
Similarly as in Observation~\ref{O4}, $b W \cap \Gamma$ has at most one component since otherwise  $(b W \cup \Gamma) \setminus (b W \cap \Gamma)$ contains
a vertex which cannot be connected to $S$ by an edge. Moreover if $b W \cap \Gamma \ne \emptyset$, then it is either a vertex or an edge. In fact, because both
$bW$ and $\Gamma$ are simple cycles corresponding to distinct components of $\mathbb{C} \setminus D(T)$, 
the intersection must be a vertex or a simple path. But if it is a path of length $> 1$, then the \emph{interior(middle)} vertices 
cannot be connected to $S$ by an edge, so we have a contradiction. In summary, we conclude that $\Gamma$ is a simple cycle and $bW \cap \Gamma$ contains at most two vertices.

Note that  $T$ is an induced graph, hence as in Observation~\ref{O1} 
at every vertex in $V(\Gamma)$ there is an edge lying on the right of $\Gamma$, or going toward the component containing $v$.  
Furthermore, such edges cannot have the other ends on $V(\Gamma)$ because $T$ is induced; i.e., there is no edge \emph{crossing} $\Gamma$. 
Therefore the graph induced by the vertices belonging to the same component
of $\mathbb{C} \setminus \Gamma$ as $v$ must be connected, because for example if there are two components of $G \setminus \Gamma$ contained in the same
component of $\mathbb{C} \setminus \Gamma$,
then the two components must be separated by a face, hence there must be an edge \emph{crossing} $\Gamma$ because every face of $G$  is a triangle.  
We conclude that $v$ can be connected to each vertex in $V(\Gamma)$ by a path lying in $G \setminus \Gamma$ except the ends, hence  
$v$ can be connected by a path lying in  $W \setminus K \bigl( \supset W \setminus V(bW) \bigr)$ to any other vertices in $V(S)$ because $\Gamma$ contains at least three vertices while 
we proved that $b W \cap \Gamma$ contains at most two vertices. Thus the claim, the induced graph with the vertex set $V(W) \setminus K$ is connected, has been proved.

Let $Z$ be the induced subgraph such that $V(Z) = V(W) \setminus V(b W)$. Then $Z$ is connected by the claim above, and definitely we have $S \subset Z$ and $d Z \subset V(b W)$.
In fact, $d Z = V(b W)$ because if $v \in V(b W) \subset V(b T) \subset d S$, then $v$ must have a neighbor in $S \subset Z$, and by the definition of $Z$ we have $v \notin V(Z)$.

Note that $bW$ is a simple cycle as we observed above, hence the inner boundary walk $b_i W$ is the same as the (outer) boundary walk $b W$. 
Moreover because $V(b W) = dZ$, at every vertex  $v \in V(bW)$ there is at least one edge incident to $v$ and going inside, or at least two faces incident to $v$ and lying on the left of
$bW$. Therefore the formula \eqref{innerLT} yields $\tau_i (v) \leq 1/6$, or
\begin{equation}\label{tlt}
6 \cdot \tau_i (b_i W)  \leq |V(bW)|.
\end{equation}
We also have 
\[
W_{-} := V(W) \setminus V(bW)  = V(Z) \supset V(S)
\]
and 
\[
dZ  = V(b W) \subset V(b T) \subset d S.
\]
Thus from \eqref{tlt} we have 
\begin{equation}\label{tttx}
6 \cdot \tau_i (b_i W) \leq |d Z| \leq |d S|.
\end{equation}
Now because $Z$ is a connected finite subgraph such that $Z \supset S_0$, \eqref{suff1} implies that
\begin{equation}\label{suff2}
c \cdot g(|S|) \leq c \cdot g(|Z|) \leq \sum_{v \in Z} ( \deg v - 6) = -6 \cdot \kappa(Z) = -6 \cdot \kappa(W_{-}).
\end{equation}
Therefore by GBF-2 \eqref{GBF-2} together with \eqref{tttx} and \eqref{suff2}, we have
\[
c \cdot g(|S|)  \leq -6 \cdot \kappa(W_{-}) = 6 \cdot \tau_i ( b_i W) - 6 \cdot \chi (W_{-}) \leq |d Z| - 6  \leq |d S|,
\]
because $\chi (W_{-})=1$ by the construction of $W$. Now Theorem~\ref{T:Hyp} follows from Theorem~\ref{HS2}(a),
because $S$ is an arbitrary connected finite subgraph such that $S \supset S_0$.

\section{Criteria for CP hyperbolicity}\label{S:doubt}
 The following inequality given in Theorem~\ref{T:HSit}
\begin{equation}\label{LAV2}
\mbox{lav}(G) := \sup_{S_0} \inf_{S \supset S_0} \left( \frac{1}{|S|} \sum_{v \in S} \deg v \right) > 6
\end{equation}
seems to say that the `average' of vertex degrees is greater than $6$, but we bet it says something different because in this formula one has to consider \emph{all} the sufficiently large subgraphs 
$S\supset S_0$. For example,  if $G$ contains a connected infinite subgraph such that all of its vertices are of degree at most 6,
then neither  Theorem~\ref{T:HSit}(b)  nor  Corollary~\ref{C:useless} can be applied to this graph $G$ even for the case when  $G$ is dominated by vertices of degrees  $\geq 7$.
See Figure~\ref{badline}.

 \begin{figure}[t]
\begin{center}
 \includegraphics{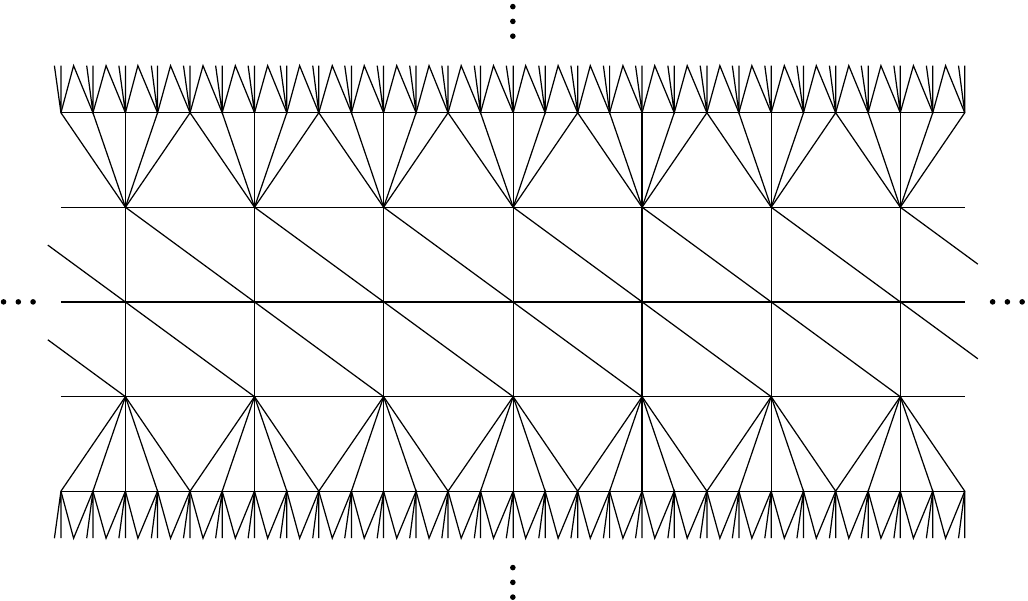}
 \caption{A disk triangulation graph including  infinitely many vertices of degree $6$ on the `middle' horizontal line. This graph is CP hyperbolic by Theorem~\ref{T:BE} or
 Theorem~\ref{T:partition}, but neither Corollary~\ref{C:useless} nor Theorem~\ref{T:HSit}(b) can be applied to this graph.}\label{badline}
\end{center}
\end{figure}

We believe that the real meaning of the inequality \eqref{LAV2} is that, in some sense, \emph{non-negatively curved clusters}, that is, the components of
 $$G \setminus \{ v \in V : \deg v \geq 7 \},$$
 should be placed  somewhere  far from a fixed spot (say $S_0$), and how far it would be must be determined by  the size of each cluster. 
 In this context what was improved in Theorem~\ref{T:Hyp} (or Corollary~\ref{C:useless}) from Theorem~\ref{T:HSit}(b)  is  
 that the graph in consideration could be still CP hyperbolic even 
when the non-negatively curved clusters are closer to a fixed spot than the inequality \eqref{LAV2} requires.

To understand this phenomenon more precisely,  let $H_n$ be the finite graph which is isomorphic to a ball of radius $n$ in the hexagonal regular triangulation (Figure~\ref{F:H_n}). 
We place them on a geodesic line, say on the real line, in order of size. Suppose  $l_n$ is the line segment between $H_n$ and $H_{n+1}$, and we triangulate the plane so that
the resulting graph $G$ contains all $H_n$'s and the geodesic real line, and $7 \leq \deg v \leq 8$ if $v \notin \bigcup_{n \in \mathbb{N}} H_n$. See Figure~\ref{F:H_n2}.

 \begin{figure}[t]
\begin{center}
 \includegraphics{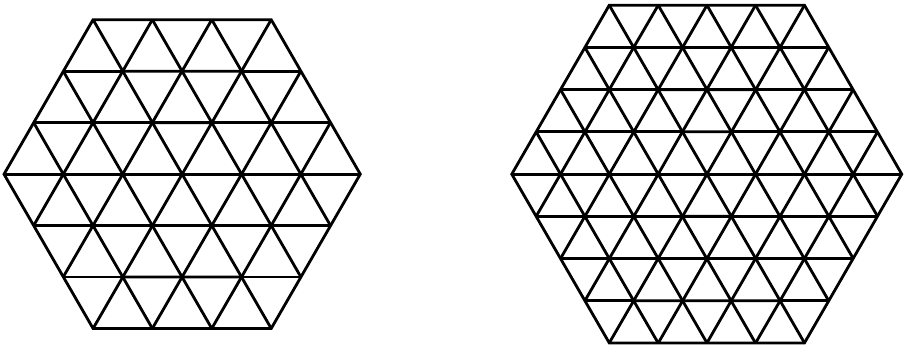}
 \caption{Graphs $H_3$ and $H_4$.}\label{F:H_n}
 \vspace{.7 cm}
 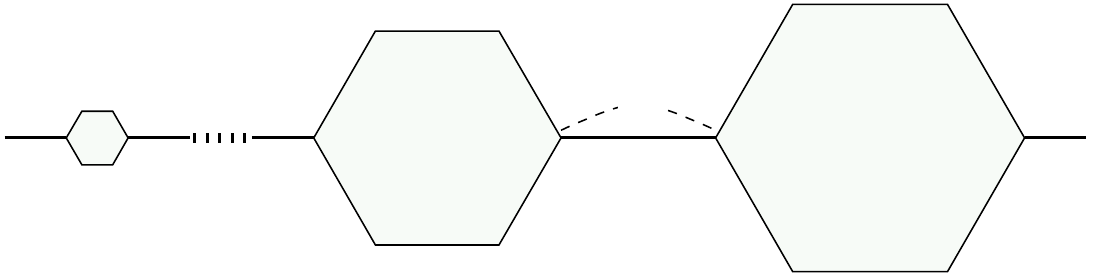
\caption{Constructing a disk triangulation graph $G$ such that $\mbox{\rm lav}(G) >6$ or $\mbox{\rm lav}(G) =6$ depending on the sequence $|l_n|$.
The regions above and below the illustration are to be triangulated so that vertices on $l_n$ and  newly added vertices are of degree $7$ or $8$, and the real line stays geodesic.}\label{F:H_n2}
\end{center}
\end{figure}

Note that $|V(H_n) | = 3n ^2 +3n +1 \approx 3 n^2$. Thus if $|l_n| \geq c n^2$ for some $c >0$, one can check that 
\[
\mbox{lav}(G) \geq  6 + \frac{c}{3+c} >6,
\]
hence $G$ is CP hyperbolic by Theorem~\ref{T:HSit}(b). On the other hand, if $c_1 n^{3\alpha-1} \leq |l_n| \leq c_2 n^{\beta}$ 
for some constants $c_1, c_2 >0$, $\alpha \in (1/2, 1)$, and $\beta \in (1/2, 2)$, then one can check that $\mbox{lav}(G)=6$, hence Theorem~\ref{T:HSit}(b) is not applicable
in this case, but 
\[
\sup_{S_0} \inf_{S \supset S_0} \frac{1}{|S|^{\alpha}} \sum_{v \in S} (\deg v - 6) \geq \frac{c_1}{3 \alpha}  >0,
\]
so Corollary~\ref{C:useless} implies that $G$ is CP hyperbolic. (Thus it is enough to have $|l_n| \approx n$ for $G$ to be CP hyperbolic.)

As we saw in the example given by Figure~\ref{badline}, none of Theorems~\ref{T:HSit}(b) and  \ref{T:Hyp} can be applied to graphs with an infinite non-negatively curved cluster.
To compensate for this weakness we will give two more criteria for CP hyperbolicity, which will be derived from some previously known results. To explain them, however, 
we need to introduce a new concept, \emph{isoperimetric constants}.

Let $G$ be a tessellation in the plane, and $S$ a subgraph of $G$. Recall that $\partial S$, $d S$, and $b S$ denote the edge boundary, the vertex boundary, and the boundary walk,
respectively, of $S$, which were defined in Section~\ref{prelim}. We also define the \emph{inner vertex boundary} $d_0 S$ by the set of vertices in $V(S)$ that have neighbors
in $V \setminus V(S)$. Note that $d_0 S \subset V(S)$ while $d S \cap V(S) = \emptyset$, and every edge in $\partial S$ must connect $d_0 S$ to $d S$. Now we define the \emph{isoperimetric constants}
$\imath(G)$, $\imath^*(G)$,  $\jmath(G)$, and  $\jmath_0 (G)$  by
\begin{equation}\label{isoconst}
\begin{aligned}
& \imath (G) = \inf_S \frac{|\partial S|}{\mbox{\rm{Vol}}(S)}, \qquad  & \imath^{*} (G) = \inf_S \frac{|bS|}{|F(S)|}, \\
& \jmath (G) = \inf_S \frac{|dS|}{|V(S)|},  & \jmath_0 (G) = \inf_S \frac{|d_0S|}{|V(S)|}, 
\end{aligned}
\end{equation}
where $\mbox{\rm{Vol}}(S) = \sum_{v \in V(S)} \deg v$, and the infima are taken over all nonempty finite subgraphs $S \subset G$ (with $F(S) \ne \emptyset$ for $\imath^{*} (G)$).
Also note that $|b S|$ is not the number of vertices in $bS$ but the number of edges as defined before Lemma~\ref{L:S_n}.
The constants $\imath(G), \jmath(G)$(or $\jmath_0(G)$), and $\imath^{*}(G)$  characterize some properties of the edge set,
the vertex set, and the face set, respectively, of $G$, and are discrete analogues of Cheeger's constant \cite{Che70}. We also say that a strong isoperimetric inequality holds on $G$  
if one of the constants in \eqref{isoconst} is positive.

\begin{lemma}\label{twoveriso}
For any tessellation $G$ in the plane, we have
\begin{equation}\label{j}
\jmath_0 (G) =\frac{\jmath(G)}{1+ \jmath (G)}.
\end{equation}
In particular,  $\jmath(G)>0$ if and only if $\jmath_0(G)>0$.
\end{lemma}
\begin{proof}
If $\alpha > \jmath(G)$, there exists $S \subset G$ such that $|dS|< \alpha |V(S)|$. Let $T$ be the induced subgraph such that $V(T) = V(S) \cup dS$. Then definitely
$d_0 T \subset d S$, hence
\[
|V(T)| = |V(S)| + |d S| \geq \frac{|d S|}{\alpha}+|dS|\geq \frac{1+\alpha}{\alpha} \cdot |d_0 T|. 
\]
Therefore we have $\jmath_0 (G) \leq \jmath(G)/(1+ \jmath(G))$ since $|d_0 T| /|V(T)| \geq \jmath_0 (G)$ and $\alpha$ is an arbitrary positive number such that 
$\alpha > \jmath(G)$.

Now if $\jmath(G) =0$, the above computation shows that $\jmath_0(G)=0$, hence \eqref{j} holds. 
Suppose $\jmath(G) >0$, and let $S$ be a nonempty finite subgraph of $G$.
Let $T$ be the induced subgraph of $G$ such that $V(T) = V(S) \setminus d_0 S$. Then we have $d T \subset d_0 S$, and
\[
|V(S)| = |V(T)|+ |d_0 S| \leq \frac{1}{\jmath(G)} |dT| + |d_0 S| \leq \frac{1 + \jmath(G)}{\jmath(G)} \cdot |d_0 S|.
\]
Thus we have $\jmath(G)/(1+ \jmath(G)) \leq \jmath_0(G)$ since $S \subset G$ is arbitrary. The lemma follows.
\end{proof}

We also have
\begin{theorem}\label{T:dual}
Let $G$ be a tessellation in the plane and  $G^*$ be the dual graph of $G$. Then the following hold.
\begin{enumerate}[(a)]
\item $\imath(G^*) >0$ if and only if $\imath^{*}(G)>0$;
\item $\jmath_0 (G)>0$ if and only if $\jmath_0 (G^*)>0$;
\item if $\jmath_0 (G)>0$, then $\imath(G)>0$ and $\imath^*(G)>0$.
\end{enumerate}
\end{theorem}
\begin{proof}
This theorem was proved in \cite[Theorem~1]{Oh14}. See also  \cite{OS16}, where the above statements were extended to planar graphs with more than one end.
In fact, planar graphs studied in \cite{Oh14}  are those whose dual graphs are simple (i.e., dual graphs do not have self loops), so every edge in the given graph
has to be incident to two distinct faces.
Furthermore, in \cite{Oh14} a simple infinite planar graph was called \emph{proper} if every face is homeomorphic to a closed disk (i.e., the facial walk of  each face is a simple cycle of length at least three), 
and a proper infinite planar graph was called \emph{normal} if every two faces are either disjoint or intersect in one vertex or one edge. 
Therefore tessellations in this paper must be normal graphs in \cite{Oh14}, which are definitely proper by definition. Thus  \cite[Theorem~1]{Oh14} can
be interpreted as above.
\end{proof}

An easy consequence of Theorem~\ref{T:dual} is the following.

\begin{lemma}\label{boundedfaces}
Suppose $G$ is a tessellation in the plane such that $\deg f \leq K$ for all $f \in F(G)$, where $K$ is a fixed positive constant. If $\imath^{*} (G) >0$, then $\jmath_0 (G)>0$.
\end{lemma}
\begin{proof}
Suppose $\imath^{*} (G) >0$. Then $\imath(G^*)>0$ by Theorem~\ref{T:dual}(a), hence there exists $c >0$ such that $\sum_{ v \in V(S)} \deg v = \mbox{\rm Vol} (S)  \leq c |\partial S|$ for every $S \subset G^*$.
Then since $\deg v \leq K$ for every $v \in V(G^*) = F(G)$, we have
\[
|V(S)| \leq \sum_{v \in V(S)} \deg v \leq c |\partial S| \leq cK |d_0 S|,
\]
showing that $\jmath_0 (G^*) \geq 1/cK >0$. Now the lemma follows from Theorem~\ref{T:dual}(b).
\end{proof}

Now we are ready to provide some more criteria for hyperbolicity.

\begin{theorem}\label{T:BE}
Suppose $G$ is a disk triangulation graph such that $\deg v \geq 6$ for every $v \in V$. If there exists $K\in \mathbb{N}$ 
such that  every  combinatorial ball $B_K (v)$ of radius $K$ contains a vertex of degree at least $7$, then $G$ is CP hyperbolic and transient.
\end{theorem}

Theorem~\ref{T:BE} is in fact an easy consequence of a result in  \cite{BE00} by Bonk and Eremenko, where they proved the following.

\begin{theorem}[Bonk and Eremenko]\label{T2:BE}
Let $M$ be an open simply connected non-positively curved Aleksandrov surface.  Then the following two statements are equivalent.
\begin{enumerate}[(i)]
\item There exist $R > 0$ and $\epsilon > 0$ such that every relatively compact open disk $D_R (a) \subset M$ has integral curvature less than $-\epsilon$,
where $D_R (a)$ is the disk in $M$ with radius $R$ and centered at $a \in M$.
\item A linear isoperimetric inequality holds on $M$. That is, there exists $c >0$ such that all Jordan regions $\Omega \subset M$ satisfy 
\begin{equation}\label{E:isoconst}
\mbox{\rm Area}(\Omega) \leq c \cdot \mbox{\rm length} (b \Omega),
\end{equation}
where $b \Omega$ is the topological boundary of $\Omega$.
\end{enumerate}
\end{theorem} 

Bonk and Eremenko  proved in the same paper that the above two statements are also equivalent to both  Gromov hyperbolicity and  tightness of $M$.
Here we omitted the last two equivalent statements since we do not need them in this paper.

\begin{proof}[Proof of Theorem~\ref{T:BE}]
The polyhedral surface $\mathcal{S}_G$ we defined in Section~\ref{S:CGB} is an open simply connected complete Aleksandrov surface which is locally Euclidean except at the points
corresponding to vertices $v \in V$, where it carries an atomic curvature $2 \pi \kappa(v)$ as we saw in \eqref{curvertex}. Moreover, because 
$\deg v \geq 6$ for all $v \in V$, $\mathcal{S}_G$ is non-positively curved. Let $d_G$ and $d_\mathcal{S}$ be the metrics on $G$ and $\mathcal{S}_G$, respectively, and
$h : G \to \mathcal{S}_G$  be the natural embedding. Then it is not difficult to see that there exist $L>0$ and $C>0$ such that
\begin{enumerate}[(a)]
\item $\frac{1}{L} \cdot d_G (v, w) \leq d_\mathcal{S} \bigl( h(v), h(w) \bigr) \leq L \cdot d_G (v, w)$ for all $v, w \in V$, and
\item for all $a \in \mathcal{S}_G$, there exists $v \in V$ such that $d_\mathcal{S} \bigl( a, h(v) \bigr) \leq C$.
\end{enumerate}
Therefore there exists $R>0$ such that every disk $D_R (a) \subset \mathcal{S}_G$ of radius $R$ includes the image of a combinatorial ball 
$B_K (v) \subset G$ of radius $K$. Now because $B_K (v)$ contains a vertex of degree at least $7$ and $\mathcal{S}_G$ is non-positively curved, the integral curvature
of $D_R (a)$ will be less than or equal to $2 \pi \cdot (-1/6) = - \pi/3$. Thus by Theorem~\ref{T2:BE}, a linear isoperimetric inequality holds on $\mathcal{S}_G$.

Let $S$ be a nonempty subgraph of $G$. Here we want to show that $|F(S)| \leq c_0 |b S|$ for some absolute constant $c_0 > 0$, so by adding to $S$ 
all the finite components of $G \setminus S$, removing from $S$ all the edges and vertices that are not incident to faces in $F(S)$, 
and then considering each component of $D(S)^\circ$ separately, we may assume that $D(S)$ consists of faces and $D(S)^\circ$ is simply connected. 
Now let $\Omega$ be the region in $\mathcal{S}_G$ that corresponds to $D(S)$. Then because $\Omega$ is simply connected and every face in $F(G)$ corresponds to
an equilateral triangle of side length $1$, we have
\[
\frac{\sqrt{3}}{4} \cdot |F(S)| = \mbox{\rm Area}(\Omega) \leq c \cdot \mbox{\rm length}(\partial \Omega) = c \cdot |b S|,
\]
where $c$ is the constant in \eqref{E:isoconst}. Thus $\imath^* (G) \geq \sqrt{3}/4c >0$ because $S \subset G$ is arbitrary, and by Lemmas~\ref{twoveriso} and \ref{boundedfaces}
we see that $\jmath(G) >0$. We conclude that $G$ is CP hyperbolic by Theorem~\ref{HS2}(a), and then $G$ is transient by Theorem~\ref{HS1}.
\end{proof}

The above arguments do not hold if there exists $v \in V$ such that $\deg v <6$, because the surface $\mathcal{S}_G$ has to be non-positively curved.
 (In fact, $G$ is allowed to have two vertices of degree 5 or one vertex of degree 4,
but nevertheless Theorem~\ref{T:BE} cannot be proved using above arguments if there are many vertices of degree at most $5$. See \cite[p.\ 64]{BE00}.)  
In the next criterion $G$ is allowed to have infinitely many vertices of degree at most $5$, but the formulation is a little bit more complicated.

\begin{theorem}\label{T:partition}
Suppose $G$ is a disk triangulation graph. If there exist $\epsilon >0$, $K>0$, and  a partition $\{ \Gamma_\alpha : \alpha \in I \}$ of $V$ such that for all $\alpha \in I$ we have
\begin{enumerate}[(a)]
\item $\kappa(\Gamma_\alpha) \leq - \epsilon$ and
\item $|V(\Gamma_\alpha)| \leq K$,
\end{enumerate}
then $G$ is CP hyperbolic and transient. Here the letter $I$ denotes an index set, and  $\{ \Gamma_\alpha \}$ being a partition of $V$ means that 
each $\Gamma_\alpha$ should be connected in $G$, and we have $\bigcup_{\alpha \in I} V(\Gamma_\alpha) = V(G)$ and $V(\Gamma_\alpha) \cap V(\Gamma_\beta) = \emptyset$ 
for $\alpha \ne \beta$.
\end{theorem}
\begin{proof}
It was shown in \cite[Theorem~6]{OS16} that  if $G$ is a tessellation  satisfying the conditions (a) and (b), then we have $\jmath_0 (G) >0$. Therefore $\jmath(G)>0$ by Lemma~\ref{twoveriso},
hence $G$ is CP hyperbolic and transient as in the proof of Theorem~\ref{T:BE}.
\end{proof}

The condition (a) above actually means that for each $\alpha$, there should be a lot of edges in $E(\Gamma_\alpha) \cup \partial \Gamma_\alpha$
for the size of $\Gamma_\alpha$. In fact, we have
\[
\kappa(\Gamma_\alpha) = \sum_{v \in V(\Gamma_\alpha)} \frac{6 - \deg v}{6} = |V(\Gamma_\alpha)| - \frac{1}{6} \sum_{v \in V(\Gamma_\alpha)} \deg v,
\]
so the condition (a)  can be restated as {\it
\begin{enumerate}[(a$'$)]
\item the average vertex degree of $\Gamma_\alpha$ is greater than $6$; i.e., $$\frac{1}{|V(\Gamma_\alpha)|} \sum_{v \in V(\Gamma_\alpha)} \deg v >6.$$
\end{enumerate}
}
\noindent This is because  $6 \cdot \kappa(v) \in \mathbb{Z}$ for every $v \in V$.

One can see that either of Theorem~\ref{T:BE} or Theorem~\ref{T:partition} implies that the example in Figure~\ref{badline} is CP hyperbolic, 
but Corollary~\ref{C:useless} cannot be applied to this graph as we saw before.
On the other hand, the example in Figure~\ref{F:H_n2} can be shown to be CP hyperbolic by Corollary~\ref{C:useless}, depending on the
lower bounds of $|l_n|$, but no isoperimetric constant in \eqref{isoconst} is positive for this graph because it contains $H_n$ for all $n \in \mathbb{N}$. Therefore neither of
Theorems~\ref{T:BE} and \ref{T:partition} can be applied to this graph.

\section{Vertex extremal length and square tilings}\label{S:VEL}
In this section we will discuss vertex extremal length of a graph, one of the main tools used by He and Schramm in \cite{HS95}. Suppose a connected graph $G=(V, E)$ is given.
A function $\mu : V \to [0, \infty)$ is called a \emph{metric} defined on $V$, and the area of $\mu$ is defined by
\[
\mbox{\rm area}_\mu (V) = \| \mu \|_V^2 := \sum_{v \in V} \mu(v)^2.
\]
The metric $\mu$ is called \emph{admissible} if $0<\| \mu \|_V < \infty$, and $\mathcal{M}(V)$ denotes the set of admissible metrics on $V$.
For a simple path  $\gamma$ in $G$,  we define the $\mu$-\emph{length} of $\gamma$  by 
\[
L_\mu (\gamma) := \sum_{v \in V(\gamma)} \mu (v).
\]
If $\Gamma$ is a collection of simple paths in $G$, we let 
$$L_\mu (\Gamma) := \inf_{\gamma \in \Gamma} L_\mu (\gamma),$$ 
and we define the \emph{vertex extremal length}(VEL) of $\Gamma$ by
\[
\mbox{VEL}(\Gamma) := \sup_{\mu \in \mathcal{M}(V)} \frac{L_\mu (\Gamma)^2}{\mbox{\rm area}_\mu (V)} =
\sup_{\mu \in \mathcal{M}(V)} \left\{ \inf_{\gamma \in \Gamma} \frac{L_\mu (\gamma)^2}{\mbox{\rm area}_\mu (V)} \right\}.
\]
Finally for disjoint subsets $A, B \subset V$,  we let $\Gamma(A, B)$ be the collection of simple paths in $G$ connecting $A$ to $B$. Then the \emph{vertex extremal distance}
between $A$ and $B$ is defined by 
$\mbox{\rm VEL}(A, B) := \mbox{VEL}(\Gamma(A, B)).$
 Here we allow  $B = \infty$ if $G$ is an infinite graph, but in this case $\Gamma(A, \infty)$ should be  the collection of simple paths that start at vertices in $A$ and 
 eventually leave every finite subset of $V$. Definitely every $\gamma\in \Gamma(A, \infty)$ must be an infinite path of the form $\gamma = [a_0, a_1, a_2, \ldots]$. 
We also use the notation $\mbox{\rm VEL}(v, \infty)$ and $\Gamma(v,\infty)$ for $\mbox{\rm VEL}(\{v \}, \infty)$ and $\Gamma(\{v\},\infty)$, respectively.

Vertex extremal length was introduced by Cannon \cite{Can94} and is a discrete analogue of the classical extremal length (cf.\ \cite{Ahlfors73, Lehto}).
Another discrete analogue is the one called \emph{edge extremal length}(EEL), which was introduced earlier than VEL by Duffin \cite{Duf62}. It is well known that
EEL is closely related to the properties of electric network and random walk, 
whereas VEL is related to circle packing types. In fact, it was VEL and EEL properties that were extensively used in the proof  of Theorem~\ref{HS1}.

Below are listed some properties of VEL that are needed later in this paper.
\begin{theorem}\label{list}
For a connected graph $G=(V,E)$ and disjoint subsets $A, B \subset V$, the following hold.
\begin{enumerate}[(a)]
\item If $\mbox{\rm VEL}(v, \infty) = \infty$ for some $v \in V$, then $\mbox{\rm VEL}(A, \infty) = \infty$ for every finite $A \subset V$.
In this case the graph $G$ is called \emph{VEL parabolic}, and otherwise  $G$ is called \emph{VEL hyperbolic}.
\item If $V$ is finite, or if there are only finitely many vertices between $A$ and $B$, then there exists $\mu_0 \in \mathcal{M}(V)$ such that 
$\mbox{\rm VEL}(A, B) = L_{\mu_0} (\Gamma(A,B))^2/\mbox{\rm area}_{\mu_0}(V)$. Such $\mu_0$ is called an extremal metric, and extremal metrics are
unique up to multiplication by positive constants.
\item $\mbox{\rm VEL}(v, \infty) = \infty$ if and only if there exists $\mu_0 \in \mathcal{M}(V)$ such that $L_{\mu_0} (\gamma) = \infty$ for every $\gamma \in \Gamma(v,\infty)$.
Such $\mu_0$ is also called an extremal metric.
\item If $G$ is a disk triangulation graph, then $G$ is CP parabolic if and only if $G$ is VEL prabolic.
\item Suppose $\phi : H \to G$ is a graph homomorphism such that for some $K >0$, we have $|\phi^{-1}(\{v \}) | \leq K$ for every $v \in V$; i.e., we assume that every $v \in V$ 
has at most $K$ preimages. Then if $G$ is VEL parabolic, so is $H$. In particular, if $G$ includes a VEL hyperbolic subgraph, then $G$ itself must be VEL hyperbolic.
\end{enumerate}
\end{theorem}
\begin{proof}
See \cite{HS95, Sch93}  for (b)--(d). The proofs for (a) and (e) are simple and left to the reader.
\end{proof}

We next consider square tilings studied in \cite{CFP94} and \cite{Sch93}.  
Let $R$ be a rectangle $[0,1]\times [0, h] \subset \mathbb{R}^2$ or a ring(cylinder without top and bottom) $S^1 \times [0, h] \subset \mathbb{R}^3$, where $h$ is a positive number and $S^1$ is the circle in the $xy$-plane with circumference $1$ and centered at the origin.
A \emph{square tiling} $\mathcal{T}$ of $R$ is a finite collection of squares $\{ Q_v : v \in V \}$ such that $\bigcup_{v \in V} Q_v = R$
and for $v \ne w$ the squares $Q_v$ and $Q_w$ have disjoint interiors. Here every square $Q_v$ must have sides parallel to axes if $R$ is a rectangle, 
or parallel to either the $xy$-plane or the $z$-axis if $R$ is a ring. If $R$ is a ring, moreover,
 a square should be interpreted as a square on the ring $R$, not a square in $\mathbb{R}^3$, hence if a side is  parallel to the $xy$-plane then it is not 
 a line segment but a circular arc. 
 We further assume that no four distinct squares in $\mathcal{T}$ have nonempty intersections, and define the contact graph $G=(V, E)$ of $\mathcal{T}$ as we did for circle packings.
 That is, $G$ is defined so that the vertex set $V$ is just the index set of the square tiling, and there is an edge $[v,w] \in E$ for $v, w \in V$,
 $v \ne w$, if and only if $Q_v \cap Q_w \ne \emptyset$. In these assumptions  one can check that $G$
 is a triangulation graph of a closed (topological) disk or annulus, depending on whether $R$ is a rectangle or a ring, respectively. Also see \cite[Observation~1.2]{Sch93}.
 
 Let $R_b$ and $R_t$ be the bottom and the top sides of $R$. That is, let $R_b = [0,1]\times \{0\}$ and $R_t =[0,1]\times \{h \}$ if $R$ is a rectangle, 
 and  $R_b = S^1 \times \{0\}$ and $R_t =S^1\times \{h \}$ if $R$ is a ring.
 Now we define $A = \{ v \in V : Q_v \cap R_b \ne \emptyset \}$ and $B = \{ v \in V : Q_v \cap R_t \ne \emptyset \}$.
 \begin{lemma}\label{lemma1}
 Let $G$, $A$, $B$ be as above. Then we have $\mbox{\rm VEL}(A,B) = h$ with an extremal metric 
 \[
 \mu_0 (v) =  \mbox{\rm the side length of the square } Q_v.
 \]
 Moreover, any other  extremal metric must be of the form $k \cdot \mu_0 (v)$ for some $k >0$.
 \end{lemma}
 \begin{proof}
 The facts that $\mbox{\rm VEL}(A,B) = h$ and $\mu_0$ is an extremal metric  can be proved with a small modification of the arguments in \cite[Lemma~4.1]{Sch93}. 
 The proof for the uniqueness part is given in \cite[Lemma~3.1]{Sch93}. We leave the details to the reader.
 \end{proof}
 
 \section{Graphs without cut locus and proof of Theorem~\ref{T:S}}\label{S:des}
Let $G=(V,E)$ be a disk triangulation graph and suppose $w \in V$. Following \cite{BP01, BP06}, we define the \emph{cut locus} of $w$ as
\[
C(w) := \{ v \in V : d(w, v') \leq d(w, v) \mbox{ for all neighbors } v' \mbox{ of } v \}.
\]
Therefore $C(w) \subset V$ is the set where the function $v \mapsto d(w, v)$ attains its local maxima. Also note that $C(w) = \emptyset$ if and only if 
every $v \in V$ has a neighbor $v'$ such that $d(w, v') = d(w, v) + 1 > d(w, v)$. If $C(w) = \emptyset$  for all $w \in V$, we will say that $G$ has \emph{no cut locus}.

The definition of cut locus might be confusing because it has nothing to do with the number of shortest paths. That is, a vertex $v$ could belong to $C(w)$
even though there is only one shortest path connecting $v$ and $w$, and it is possible to have $v \notin C(w)$ even though there are more than one shortest paths between $v$ and $w$.
See Figure~\ref{F:cutlocus}.
 \begin{figure}[t]
\begin{center}
 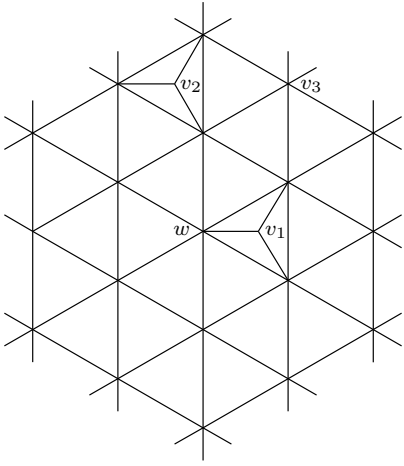
 \caption{Cut locus of $w$. The vertices $v_1$ and $v_2$ belong to $C(w)$ even though there is only one shortest path from $w$ to each of these vertices,
 while $v_3 \notin C(w)$ even though there are two shortest paths from $w$ to $v_3$.}\label{F:cutlocus}
 \end{center}
\end{figure}

Fix $v_0 \in V$, and suppose $C(v_0) = \emptyset$. Then because every $v \in S_n = S_n (v_0)$ has a neighbor $v' \in S_{n+1}$, it is not difficult to see that
$B_n = B_n (v_0)$ is simply connected for every $n \in \mathbb{N}$; i.e., we have $m =1$ and $b B_n = b A_n = \Gamma_1$ in the computations performed in Section~\ref{S:T1}. 
Moreover, no vertices in $S_n$ can lie inside $b B_n$, so we must have $|b B_n| = |S_n|$. 
Finally it is also true that $|S_{n+1}| = | b B_n | + \mbox{(total number of extra edges)}$,
since otherwise there would be a vertex in $S_{n+1}$ lying inside  $b B_{n+1}$. 
Therefore the proof in Section~\ref{S:T1} actually shows that  \eqref{spheresize} becomes an equality in this case; i.e., we have
\[
|S_{n+1}|  =  \sum_{j=0}^n (k_j +6) =: a_{n+1}
\]
for every $n =0,1,2, \ldots$, where $\{ k_j \}$ is the degree excess sequence defined in Theorem~\ref{Rep}.

Baues and Peyerimhoff showed in  \cite{BP06} (cf.\ \cite{BP01, Lyn67}) that if the \emph{corner curvature} 
\[
\kappa (v, f) = \frac{1}{\deg v} + \frac{1}{\deg f} - \frac{1}{2}
\]
is non-positive for all pairs $(v, f) \in V \times F$ such that $v$ and $f$ are incident to each other, then $G$ has no cut locus. But in our case  $G$ is a disk triangulation graph,
hence corner curvatures become
\[
\kappa (v, f) =  \frac{1}{\deg v}  - \frac{1}{6},
\]
which is less than or equal to zero if and only if $\deg v \geq 6$. Thus we obtain the following theorem, which might be interesting by itself.

\begin{theorem}\label{degree6}
Suppose $G$ is a disk triangulation graph such that the  function $v \mapsto d(v_0, v)$ has no local maxima (i.e., $C(v_0) = \emptyset$). Then we have
\[
|S_n (v_0)| =  \sum_{j=0}^{n-1} (k_j +6) = a_n \quad \mbox{and} \quad |B_n (v_0)| = 1+ \sum_{k=1}^n a_k
\]
for every $n =1,2, \ldots$. In particular, the above formulae hold if $\deg v \geq 6$ for every vertex $v \in V$.
\end{theorem}

For example, if $G$ is the regular hexagonal triangulation of the plane, then we have $k_n=0$ for all $n \in \mathbb{N}$ and obtain
$|S_n| = 6n$ and $|B_n| = 1+3n(n+1)= 3n^2 + 3n +1$ for $n \in \mathbb{N}$, as we computed in the example given by Figure~\ref{F:H_n2}.
 Now we are ready to construct a graph needed for Theorem~\ref{T:S}, which is essentially the same as the graph constructed in \cite[Theorem~6.3]{Wood09}.

\begin{proof}[Proof of Theorem~\ref{T:S}]
Suppose $\{ k_n \}_{n=0}^\infty$ is a sequence of integers satisfying the inequality
\[
 a_n = \sum_{j=0}^{n-1} (k_j +6) \geq 3
 \]
for $n \geq 1$, and define $\mathcal{B}_n$ as the ring made up of $a_n$ squares of side lengths $1/a_n$, 
which are placed in $\mathcal{B}_n$ as one layer. See Figure~\ref{F:band}. 
 \begin{figure}[t]
\begin{center}
 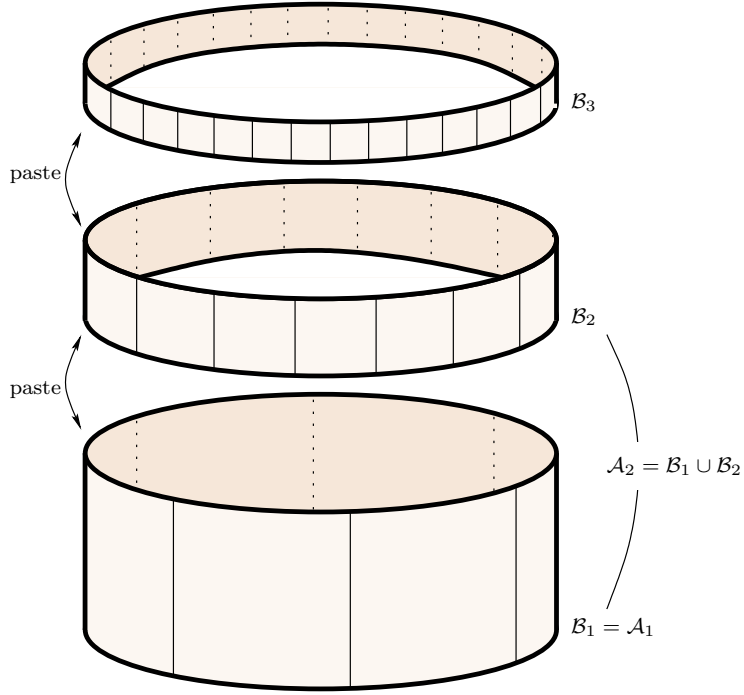
 \caption{Pasting rings made up of squares.}\label{F:band}
 \end{center}
\end{figure}
We then paste the top side(circle) of $\mathcal{B}_1$ to the bottom side(circle) of $\mathcal{B}_{2}$ so that 
no four distinct squares in $\mathcal{B}_1 \cup \mathcal{B}_2$ have nonempty intersections. 
Set $\mathcal{A}_1 = \mathcal{B}_1$, and we define $\mathcal{A}_2$ as the ring obtained by pasting $\mathcal{B}_1$ and $\mathcal{B}_2$. Next we paste the top side of $\mathcal{A}_2$ to the bottom side of $\mathcal{B}_3$ such that
no four distinct squares in $\mathcal{A}_2 \cup \mathcal{B}_3$ have nonempty intersections, and obtain a new ring $\mathcal{A}_3$.

We repeat this process and get a sequence of rings $\mathcal{A}_1 \subset \mathcal{A}_2 \subset \cdots$. Let $G_n$ be the contact graph of $\mathcal{A}_n$. Then definitely we can embed $G_n$
into the plane so that $G_1 \subset G_2 \subset \cdots$. Now by adding a vertex $v_0$ to the graph $\bigcup_{n=1}^\infty G_n$ 
and connecting $v_0$ to the vertices corresponding to the squares in $\mathcal{B}_1$, we get a disk triangulation graph, say $G$.
From the construction  the function $v \mapsto d(v_0, v)$ cannot have local maxima, and we have  $|S_n (v_0)| = a_n$. 
Therefore Theorem~\ref{degree6} implies that the sequence $\{ k_n \}_{n=0}^\infty$ must be the degree excess sequence of $G$; i.e., $k_n = \sum_{v \in B_n} ( \deg v - 6)$.

Now suppose that $\sum 1/a_n < \infty$, and we claim that the graph $G$ constructed above is CP hyperbolic in this case. 
Note that $\mbox{\rm VEL} (S_1, S_n) = \sum_{k=1}^n 1/a_k$ by Lemma~\ref{lemma1} and our construction, hence
\[
\mbox{\rm VEL}(S_1, \infty) = \lim_{n \to \infty}  \mbox{\rm VEL}(S_1, S_n) = \sum_{n=1}^\infty \frac{1}{a_n} < \infty.
\]
We conclude that $G$ is VEL hyperbolic, hence CP hyperbolic by Theorem~\ref{list}(d). 

Next let us assume $\sum 1/(a_n + a_{n+1})< \infty$, and we want to show that $G$ is transient in this case. But by the simple criterion of transience by T.~Lyons
\cite{LyT83} (cf.\ \cite[Theorem~2.11]{LP16} or \cite[Theorem~3.33]{Soa}), it suffices to show that there exists a unit flow from $v_0$ to infinity with finite energy.
For this purpose we define a \emph{water flow}, an antisymmetric function $\theta: V \times V \to \mathbb{R}$, as follows. 
Suppose $[v, w] \in E_n$ for some $v \in S_n$ and $w \in S_{n+1}$, $n \geq 1$.
Here $E_n$ is the set of edges connecting $S_n$ to $S_{n+1}$ as before. Let $Q_v$ be the square corresponding to $v$, which should be included in the ring $\mathcal{B}_n$
in our construction. Similarly let $Q_w \subset  \mathcal{B}_{n+1}$ be the square corresponding to $w$, and we define  the amount of water flowing from $v$ to $w$ as
\[
\theta(v, w) = - \theta(w, v) = \mbox{length}(Q_v \cap Q_w).
\]
Note that $\theta(v, w) >0$, because $w$ is a neighbor of $v$ and therefore $Q_v$ intersects $Q_w$ (in some finite length). Next for $v \in S_1$ we define
\[
\theta(v_0, v) = - \theta(v, v_0) = \frac{1}{a_1} = (\mbox{side length of the square corresponding to } v).
\]
Finally we let $\theta(v, w) =0$ if $v$ and $w$ are not neighbors, or both $v$ and $w$ belong to $S_n$ for some $n$. 

For $v \in V \setminus \{v_0\}$,  it is clear that the total amount of  water flowing into $v$ is the same as the total amount of water flowing out of $v$, both of which are 
the side length of $Q_v$. Moreover, the amount of water flowing out of $v_0$ is one and no water flows into $v_0$. Thus $\theta$ defines a flow function from $v_0$
to infinity, so we only need to show that $\theta$ has finite energy.

One can check either from \eqref{edgesize} or directly from the construction that for $n \geq 1$, $|E_n| = |S_n| + |S_{n+1}| = a_n + a_{n+1}$ 
in our case.  Moreover if $[v, w] \in E_n$, $n \geq 1$,
then because $|\theta(v, w)|$ is less than or equal to the minimum of the side lengths of $Q_v$ and $Q_w$, we have
\[
|\theta(v, w)| \leq \min \left\{ \frac{1}{a_n} , \frac{1}{a_{n+1}} \right\} = \frac{1}{\max \{ a_n, a_{n+1} \} } \leq \frac{2}{a_n + a_{n+1}}.
\]
Now since $\theta(e) = 0$ for $e \notin \bigcup_{n=0}^\infty E_n$, we have 
\begin{align*}
\sum_{e \in E} \theta(e)^2 & = \sum_{n=0}^\infty \sum_{e \in E_n} \theta (e)^2 \leq \frac{a_1}{(a_1)^2} + \sum_{n=1}^\infty \sum_{e \in E_n} \frac{4}{(a_n + a_{n+1})^2}\\
& =  \frac{1}{a_1} + \sum_{n=1}^\infty \frac{4}{a_n + a_{n+1}}< \infty.
\end{align*}
We conclude that $G$ is transient, and this completes the proof of Theorem~~\ref{T:S}.
\end{proof}

Suppose $G$ is a disk triangulation graph such that $\deg v \geq 6$ for every vertex $v \in V(G)$. If $k_n = n$ for all $n =0,1,2, \ldots$, then every vertex $v \in S_n$, $n \geq 1$,
must be of degree 6 except one vertex  of degree $7$. Note that in this case every $B_n$ is simply connected because $G$ has no cut locus, and we have
\begin{equation}\label{conj}
\sum_{n=1}^\infty \frac{1}{a_n} < \infty.
\end{equation}
 Does this guarantee CP hyperbolicity of  $G$?
 In general, if $G$ is a disk triangulation graph such that $\deg v \geq 6$ for all $v \in V$ and satisfies \eqref{conj}, does $G$ need to be CP hyperbolic?
 If not, under what conditions can we say that $G$ is CP hyperbolic?
 Note that the graph constructed in the proof of Theorem~\ref{T:S} does not answer for this question, since it may have vertices of degree less than 6. 
 
 \section{Proof of Theorem~\ref{T:LCP}}\label{S:LCP}
 Suppose $G=(V, E, F)$ is a disk triangulation graph associated with a layered circle packing determined by the sequence $(h_k, d_k)$. Let $v_0 \in V$ be the distinguished (center) vertex of $G$,
 and  let $S_n = S_n (v_0)$ and $B_n = B_n (v_0)$ be as before. Also let $k_n$ be the degree excess sequence. 
 
 Note that $k_0 = \deg v_0 -6$,  which is the same as $k_m$ for $m = 1,2, \ldots, h_1 -1$ because every vertex in $S_1 \cup S_2 \cup \cdots \cup S_{h_1 -1}$ is of degree $6$.
 But every vertex in $S_{h_1}$ is of degree $6 + d_1$, so we have
 \[
 k_{h_1} - k_0 = k_{h_1}- k_{h_1 -1} =  d_1 \cdot|S_{h_1}|.
 \]
 On the other hand, because $\deg v \geq 6$ for all $v \in V$ by our assumption, Theorem~\ref{degree6} implies  that 
 \[
 |S_{h_1}| = \sum_{j=0}^{h_1 -1} (k_j +6)=\sum_{j=0}^{h_1 -1} (k_0 +6) = h_1 (k_0 +6).
 \]
From similar computations we obtain
\[
k_{h_1 + h_2} - k_{h_1} = d_2 \cdot |S_{h_1 + h_2}| 
\]
and 
\[
|S_{h_1 + h_2}| - |S_{h_1}| = \sum_{j=h_1}^{h_1 + h_2 -1}  (k_{h_1}+6) = h_2 (k_{h_1}+6).
\]
Let $\theta_n = h_1 + h_2 + \cdots + h_n$, $\delta_n = (6 + k_{\theta_n})/(6+k_0) =(6+  k_{h_1 + \cdots + h_n})/(6 + k_0)$, and $c_n = |S_{\theta_n}|/ (6+k_0)$ for $n \geq 1$. 
We also define $\delta_0 =1$. Then by the same computations as above we get the following system of recurrence relations
\[ 
\begin{cases}
\delta_0 =1, c_1 = h_1, \\
\delta_n - \delta_{n-1} = d_n c_n \mbox{ for } n \geq 1, \\
c_n - c_{n-1} = h_n \delta_{n-1} \mbox{ for } n \geq 2.
\end{cases}
\] 
Therefore
\[
\delta_n = \delta_0 + \sum_{k=1}^n (\delta_k - \delta_{k-1}) = 1 + d_1 c_1 + d_2 c_2 + \cdots + d_n c_n,
\]
hence for $n \geq 2$ we have
\[
c_n = c_{n-1} + h_n \delta_{n-1} = c_{n-1} + h_n (1 + d_1 c_1 + d_2 c_2 + \cdots + d_{n-1} c_{n-1}).
\]

Theorem~\ref{T:LCP} asserts that $G$ is CP parabolic or hyperbolic according as the series
\begin{equation}\label{E:series2}
 \sum_{n=2}^\infty \frac{\ln h_n}{d_{n-1} c_{n-1}}
\end{equation}
diverges or converges, respectively, where $c_n$ is as defined above. Here we claim that convergence of the series \eqref{E:series2} is equivalent to convergence of
\begin{equation}\label{E:series}
 \sum_{n=2}^\infty \frac{\ln (h_n+1)}{d_{n-1} c_{n-1}}.
\end{equation}
Definitely divergence of \eqref{E:series2} implies divergence of \eqref{E:series}. Thus let us assume that  \eqref{E:series2} converges. Note that
 the series $\sum 1/c_n$ must be convergent because $c_n \geq 2 c_{n-1}$ for all $n$. Therefore,
 \begin{align*}
 \sum_{n=2}^\infty  \frac{\ln (h_n+1)}{d_{n-1} c_{n-1}} & = \sum_{h_n =1} \frac{\ln (h_n+1)}{d_{n-1} c_{n-1}} + \sum_{h_n >1} \frac{\ln (h_n+1)}{d_{n-1} c_{n-1}} \\
  & \leq  \sum_{n=2}^\infty \frac{\ln 2}{c_{n-1}} + \sum_{n=2}^\infty \frac{ 2 \ln h_n}{d_{n-1} c_{n-1}} < \infty,
 \end{align*}
 as desired.
 
Now let us prove the parabolic case. Suppose $1 \leq  l \leq h_{n}$. Then 
\begin{align*}
|S_{\theta_{n-1} + l}| & =  \sum_{j=0}^{\theta_{n-1}-1} (k_j +6) +  \sum_{j = \theta_{n-1}}^{\theta_{n-1} + l-1} (k_{\theta_{n-1}} + 6) \\
          & = (6 + k_0) ( c_{n-1} + l \cdot \delta_{n-1}) \\
          & = (6 + k_0) \left\{ c_{n-1}+ l \cdot (1 + d_1 c_1 + d_2 c_2 + \cdots + d_{n-1} c_{n-1}) \right\} \\
          & \leq (6+ k_0) \cdot l \cdot (3 d_{n-1} c_{n-1}),
\end{align*}
because 
\begin{align*}
c_{n-1} & = c_{n-2} + h_{n-1} (1 + d_1 c_1 + d_2 c_2 + \cdots + d_{n-2} c_{n-2}) \\
               & \geq 1 + d_1 c_1 + d_2 c_2 + \cdots + d_{n-2} c_{n-2}.
\end{align*}
Therefore we obtain
\[
\sum_{j=\theta_{n-1}+1}^{\theta_n}  \frac{1}{|S_j|} \geq \frac{1}{3 (6 + k_0)}\cdot \frac{1}{ d_{n-1} c_{n-1}} \sum_{l=1}^{h_n} \frac{1}{l} 
\geq \frac{1}{3 (6 + k_0)}\cdot \frac{\ln (h_n +1 )}{ d_{n-1} c_{n-1}},
\]
and
\[
\sum_{j=h_1 + 1}^\infty \frac{1}{|S_j|} = \sum_{n=2}^\infty\sum_{j=\theta_{n-1}+1}^{\theta_n}  \frac{1}{|S_j|} \geq \frac{1}{3 (6 + k_0)} \sum_{n=2}^\infty 
\frac{\ln (h_n +1 )}{ d_{n-1} c_{n-1}}.
\]
We conclude by Theorem~\ref{T:RoSull}, the Rodin-Sullivan theorem, that $G$ is CP parabolic if the series in \eqref{E:series}, or in \eqref{E:series2}, diverges.

The above arguments almost prove the hyperbolic case as well. Note that we have $|S_{\theta_{n-1} + l}|  \geq (6+k_0)\cdot l \cdot d_{n-1} c_{n-1}$, hence
\[
\sum_{j=h_1 + 1}^\infty \frac{1}{|S_j|} \leq \frac{1}{6+k_0}\cdot \sum_{n=2}^\infty  \frac{1}{ d_{n-1} c_{n-1}} \sum_{l=1}^{h_n} \frac{1}{l} 
\leq \frac{2}{6+k_0} \cdot \sum_{n=2}^\infty  \frac{\ln (h_n +1)}{ d_{n-1} c_{n-1}}.
\]
Therefore the hyperbolic part of Theorem~\ref{T:LCP} follows from \cite[Theorem~3.6]{DW05}, where Dennis and Williams proved that a layered circle packing is CP hyperbolic
if $\sum 1/ |S_n|$ converges. Unfortunately, however, the arguments in \cite{DW05} can be applied only when the sequence $d_n$ is bounded, so we need a different approach
for the hyperbolic case. Our strategy is to adopt Siders' method in \cite{Si98}, but we will compute vertex extremal distance instead of electric resistance.

 \begin{figure}[t]
\begin{center}
 \includegraphics{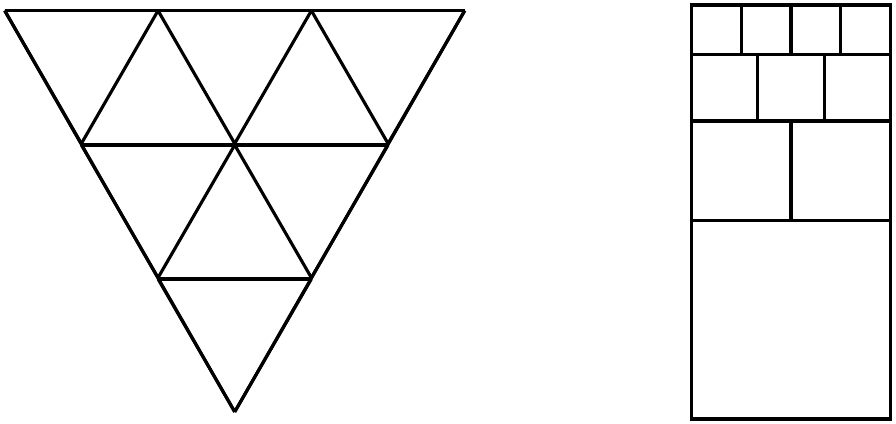}
 \caption{A triangular mesh of size $3$ (left), and the corresponding square tiling (right).}\label{F:mesh}
 \vspace{.5 cm}
  \includegraphics{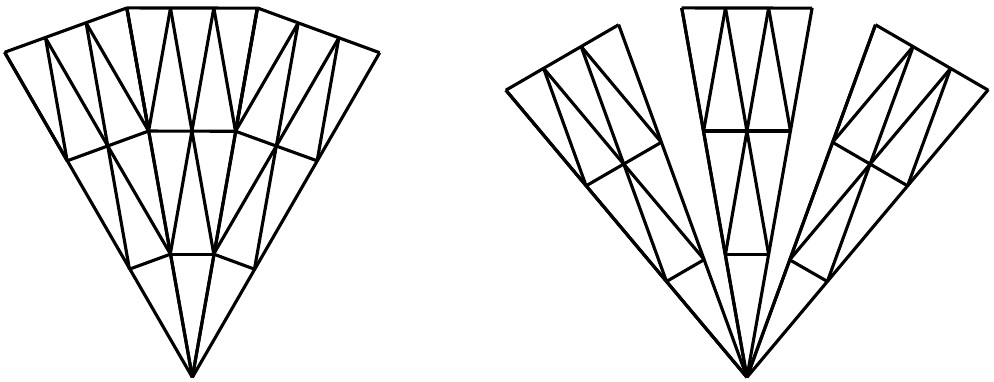}
 \caption{A 3-fold triangular mesh of size $3$ (left), and its separated triangular mesh (right).}\label{F:foldmesh}
\end{center}
\end{figure}
To adopt Siders' method precisely, we need to introduce some terminology. A \emph{triangular mesh} of size $n$, or an $n$-triangular mesh, is the graph shown
on the left of Figure~\ref{F:mesh}: it is the contact graph of the square tiling of a rectangle consisting  of $n+1$ layers such that the $k$-th layer of the tiling has $k$ squares of side
lengths $1/k$. Note that the vertex extremal distance between an apex(corner vertex) and the opposite side of an $n$-triangular mesh is $\sum_{k=1}^{n+1} 1/k$, as we discussed in 
Section~\ref{S:VEL}, and the vertex extremal distance between the 2nd layer(the set of neighbors of an apex) and the last layer(the opposite side of the apex) is $\sum_{k=2}^{n+1} 1/k$.
 A \emph{$d$-fold $n$-triangular mesh} is the graph obtained by pasting $d$ triangular meshes of size $n$ along sides so that they are pasted counterclockwise one
after another around a common vertex, the \emph{apex} (Figure~\ref{F:foldmesh}). 
Then by \emph{separating} a $d$-fold $n$-triangular mesh, we will get a graph consisting of $d$ triangular meshes of size $n$
which are connected only at the apex. Note that the hexagonal regular triangulation includes $5$-fold triangular meshes of every size, and $6$-fold triangular meshes with the first and last sides
are pasted; in this case we will say that the $n$-fold triangular mesh is \emph{closed}. 
 \begin{figure}[p]
\begin{center}
 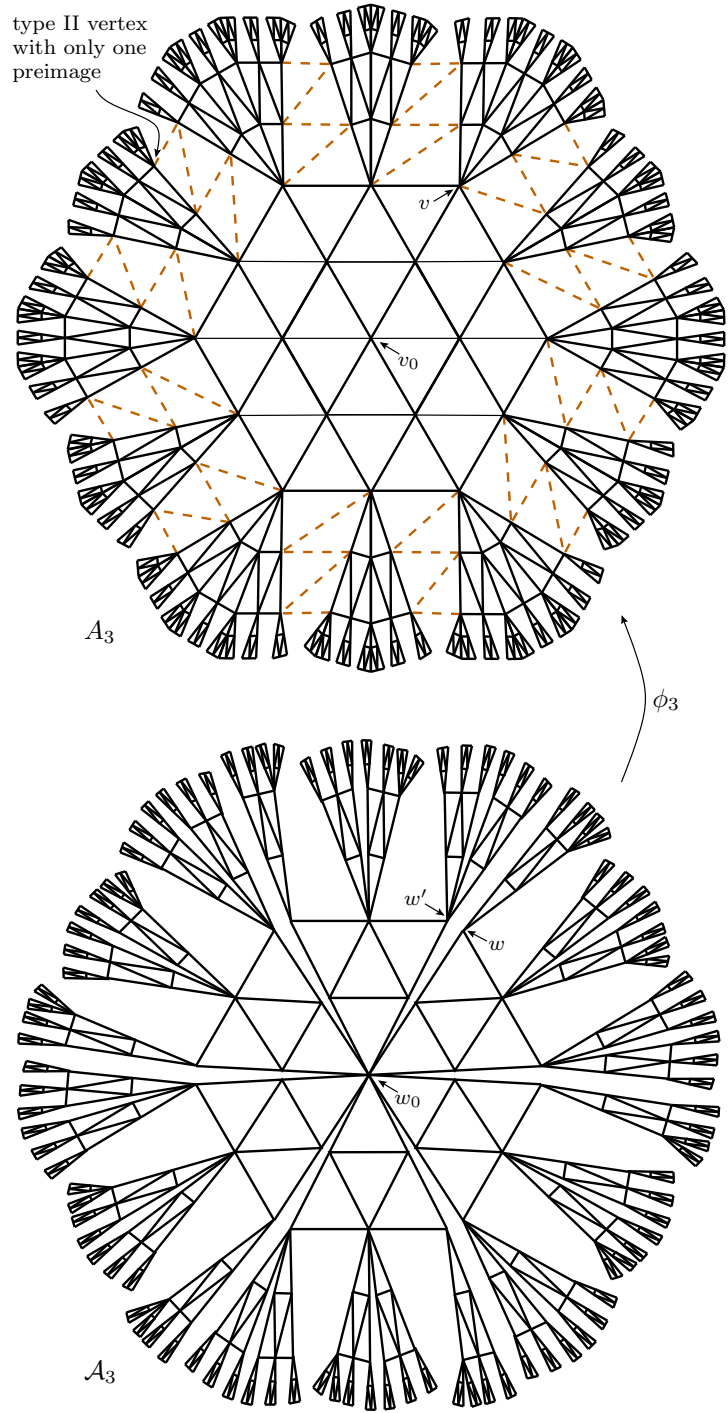
 \caption{Graphs $A_3$ (top) and $\mathcal{A}_3$ (bottom) with $\deg v_0 = 6$, $h_1 = h_2 = h_3=2$, $d_1 =2$ and $d_2 =1$. In the course of obtaining $A_2$, 
 the edges removed from $B_{\theta_2}$ are shown in the top figure as dashed lines.}\label{F:A_n}
\end{center}
\end{figure}

The disk triangulation graph $G$ in consideration includes $B_{h_1}= B_{\theta_1}$, a $(k_0+6)$-fold closed  triangular mesh of size $h_1$. We then separate this triangular mesh 
and get a graph $\mathcal{A}_1$, and denote its apex by $w_0$. Also let $A_1 = B_{\theta_1}$, $\mathcal{T}_1$  the set of vertices in $\mathcal{A}_1$ 
whose combinatorial distance from $w_0$ is $h_1$, and $\phi_1 : \mathcal{A}_1 \to A_1$ the natural graph homomorphism such that $\phi_1(w_0) = v_0$ and 
$\phi_1(\mathcal{T}_1) = S_{h_1}$. 

We next observe that $G$ contains a subgraph, say $A_2$, such that if $v \in S_{h_1}$ has two parents (that is, if $v$ is of type I as defined in Section~\ref{S:T1}), 
then a $d_1$-fold triangular mesh of size $h_2$ is attached to $v$,
and if $v \in S_{h_1}$ has only one parent (that is, if $v$ is of type II), then a $(d_1+1)$-fold triangular mesh of size $h_2$ is attached to $v$. 
We also assume that $A_2$ includes only $A_1$ and the attached triangular meshes as shown in the graph on the top  of Figure~\ref{F:A_n}, 
so that $A_2$ should be  obtained from $B_{\theta_2}$ by deleting edges. Note that the vertex sets of $A_2$ and $B_{\theta_2}$ must be the same.

If $v \in S_{\theta_1} = S_{h_1}$ has two parents, then there is only one vertex $w \in \mathcal{T}_1$ such that $\phi_1(w) = v$. In this case we attach to $w$ 
a $d_1$-fold \emph{separated} triangular mesh of size $h_2$.  If $v$ has only one parent, then there are exactly two vertices 
$w, w'  \in \mathcal{T}_1$ such that $\phi_1(w)= \phi_1(w') =v$. Suppose $w$ precedes $w'$ counterclockwise around $w_0$. 
We then attach to $w$ a (1-fold) triangular mesh of size $h_2$, and to $w'$  a $d_1$-fold \emph{separated} triangular mesh of size $h_2$.
After attaching triangular meshes to all the vertices in $\mathcal{T}_1$ as described above, we obtain a new graph $\mathcal{A}_2$ 
and a graph homomorphism $\phi_2 : \mathcal{A}_2 \to A_2$
such that $\phi_2 |_{\mathcal{A}_1} = \phi_1$ and $\phi_2 (\mathcal{T}_2) = S_{\theta_2}$, where $\mathcal{T}_2$ is the set of vertices in $\mathcal{A}_2$
whose combinatorial distance from $w_0$ is $\theta_2 = h_1+h_2$. Also note that the number of the second generation (1-fold) triangular meshes in $\mathcal{A}_2$ 
is $$d_1 \cdot |S_{\theta_1}|+ (k_0 +6) = (k_0+6)( 1+ d_1 c_1)= (k_0+6) \delta_1,$$ and that every type I vertex in $S_{\theta_2}$ has only one preimage via $\phi_2$, while
each of type II vertices in $S_{\theta_2}$ has either one preimage or two preimages.
 
 The next step is similar to the previous step, but there is a slight difference. We first observe that $G$ contains a subgraph $A_3$ with $A_2 \subset A_3 \subset B_{\theta_3}$
 such that if $v \in S_{\theta_2}$ is of type I then a $d_2$-fold triangular mesh of size $h_3$ is attached to $v$,
and if $v \in S_{\theta_2}$ is of type II then a $(d_2+1)$-fold triangular mesh of size $h_3$ is attached to $v$. Definitely we want $A_3$ to include only $A_2$ and
the attached triangular meshes, so that $A_3$ should be obtained from $B_{\theta_3}$ by deleting edges. Now if $v \in S_{\theta_2}$ is of type I,  there is only one
$w \in \mathcal{T}_2$ such that $\phi_2(w) = v$, hence we attach to $w$ a $d_2$-fold separated triangular mesh of size $h_3$. If $v \in S_{\theta_2}$ is of type II with
only one preimage $w\in \mathcal{T}_2$ such that $\phi_2(w) = v$, then we attach to $w$ a $(d_2+1)$-fold separated mesh. If $v \in S_{\theta_2}$ is of type II and there are
two elements of $\mathcal{T}_2$, say $w$ and $w'$, such that $\phi_2(w) = \phi_2 (w')= v$, then we attach to $w$ a mesh and to $w'$ a $d_2$-fold separated mesh,
assuming that $w$ precedes $w'$  counterclockwise around $w_0$. After attaching triangular meshes to all the vertices in $\mathcal{T}_2$ as described above, 
we obtain $\mathcal{A}_3$ and a graph homomorphism $\phi_3 : \mathcal{A}_3 \to A_3$
such that $\phi_3 |_{\mathcal{A}_2} = \phi_2$ and $\phi_3 (\mathcal{T}_3) = S_{\theta_3}$, where $\mathcal{T}_3$ is  defined similar to $\mathcal{T}_1$ and $\mathcal{T}_2$.
Then every type I vertex in $S_{\theta_3}$ has only one preimage via $\phi_3$ and each of type II vertices in $S_{\theta_3}$ has either one preimage or two preimages as before,
and the number of the third generation triangular meshes in $\mathcal{A}_3$ is 
 \[
 |S_{\theta_2+1}| -  |b B_{\theta_2}|  = |S_{\theta_2+1}| - |S_{\theta_2}| = (k_0+6)\delta_{2} 
 = (k_0+6)(1+ d_1 c_1 + d_{2} c_{2}).
 \]
 
 We repeat the above process to get a sequence of graphs $\mathcal{A}_1 \subset \mathcal{A}_2 \subset \mathcal{A}_3 \subset \cdots$ 
 and a sequence of graph homomorphisms $\phi_n : \mathcal{A}_n \to A_n \subset B_{\theta_n}$ such that
 $\phi_n |_{\mathcal{A}_{n-1}} = \phi_{n-1}$ and $\phi_ n (\mathcal{T}_n) = S_{\theta_n}$, where $A_n$ and $\mathcal{T}_n$ are defined similarly as in the previous steps.
 Let $\mathcal{A} = \bigcup_{n=1}^\infty \mathcal{A}_n$ and $A= \bigcup_{n=1}^\infty A_n$,
 and define $\phi : \mathcal{A} \to A$ by $\phi |_{\mathcal{A}_n} = \phi_n$. Note that $V(A) = V(G)$, and every $v \in V(A) = V(G)$ has at most two preimages via $\phi$.
Thus Theorem~\ref{list} implies that VEL hyperbolicity of $\mathcal{A}$ implies CP hyperbolicity as well as VEL hyperbolicity of $G$. Also one can check that
 the number of  $(n+1)$-th generation triangular meshes in  $\mathcal{A}$ is 
 \begin{equation}\label{E:generation}
 |S_{\theta_n +1}| -  |b B_{\theta_n}|  =  (k_0+6)(1+ d_1 c_1 + \cdots + d_{n} c_{n}).
 \end{equation}

 Let $\mu_1$ be the extremal metric for $\mbox{\rm VEL}(w_0, \mathcal{T}_1)$ such that $\mu_1 (w_0) =1$, and $T$ be a component of $\mathcal{A}_1 \setminus \{ w_0 \}$.
 Also let $X$ be the set of neighbors of $w_0$ belonging to $T$ and $Y = V(T) \cap \mathcal{T}_1$. Then $T$ is a $h_1$-triangular mesh with the apex removed, hence
  we have $|X|=2$, $|Y|= h_1+1$, and by Lemma~\ref{lemma1}
 \[
 \lambda_1 := \mbox{\rm VEL}(X, Y) = \frac{1}{2} + \frac{1}{3} + \cdots + \frac{1}{h_1+1} \leq \ln (h_1 +1).
 \]
Because every component of $\mathcal{A}_1 \setminus \{ w_0 \}$ is isomorphic to $T$, it is not difficult to see that $\mu_1 |_T$ is an extremal metric for 
$\mbox{\rm VEL}(X, Y)$. Since this argument can be applied to every component of $\mathcal{A}_1 \setminus \{ w_0 \}$, 
 Lemma~\ref{lemma1} and the fact that $\mu_1$ is an extremal metric for $\mbox{\rm VEL}(w_0, \mathcal{T}_1)$ imply that there exists $c>0$ such that
$\mu_1 (w) = c/(m+1)$  for every $w \in V(\mathcal{A}_1) \setminus \{w_0\}$ with $d_{\mathcal{A}_1} (w_0, w) = m \geq 1$.  
Note that 
\[
\mbox{\rm area}_{\mu_1} (T) := \mbox{\rm area}_{\mu_1} (V(T))  = c^2 \lambda_1. 
\]
Then because of the fact that $\mu_1$ is the extremal metric for $\mbox{\rm VEL}(w_0, \mathcal{T}_1)$ again,
the positive real number $c$ must  maximize the function
\[
f(x) = \frac{(1+ \lambda_1 x)^2}{1+ M \cdot \lambda_1 x^2}
\]
for $x>0$, where $M = (k_0 +6)$ is the number of triangular meshes in the first generation. Moreover, the maximum should be equal to $\mbox{\rm VEL}(w_0, \mathcal{T}_1)$.
Now one can easily compute that $c = 1/M$ and  $\mbox{\rm VEL}(w_0, \mathcal{T}_1) = 1 + \lambda_1/M = 1+ \lambda_1/(k_0 + 6)$.

Next let $\mu_1$ and $\mu_2$ be the extremal metrics for  $\mbox{\rm VEL}(w_0, \mathcal{T}_1)$ and  $\mbox{\rm VEL}(w_0, \mathcal{T}_2)$, respectively,
but this time we normalize them so that $\mbox{\rm area}_{\mu_1}(\mathcal{A}_1) = \mbox{\rm area}_{\mu_2}(\mathcal{A}_1) =1$.
Note that $\mu_1 (w) = \mu_1 (w')$ if $d_{\mathcal{A}_1} (w_0, w) =d_{\mathcal{A}_1} (w_0, w') $ as we observed above. Therefore for every geodesic  path $\eta$ from $w_0$
to $\mathcal{T}_1$, we must have $$L_{\mu_1} (\eta)^2 = \mbox{\rm VEL}(w_0, \mathcal{T}_1) = 1+ \lambda_1/(k_0 + 6)$$
because $\mbox{\rm area}_{\mu_1}(\mathcal{A}_1)=1$.
Moreover, because at least one $h_2$-triangular mesh is attached to each vertex in $\mathcal{T}_1$ and 
all the components of $\mathcal{A}_2 \setminus \mathcal{A}_1$ are isomorphic to each others,  one can check that we must have $\mu_2 |_{\mathcal{A}_1} = \mu_1$.
In fact, if 
\[
\inf_{\eta \in \Gamma(w_0, \mathcal{T}_1)} L_{\mu_2}(\eta) =L_{\mu_2} \bigl(\Gamma(w_0, \mathcal{T}_1) \bigr) 
< L_{\mu_1} \bigl(\Gamma(w_0, \mathcal{T}_1) \bigr)=\mbox{\rm VEL}(w_0, \mathcal{T}_1)^{1/2},
\]
 then clearly we have 
\[
 \frac{L_{\mu_2} \bigl(\Gamma(w_0, \mathcal{T}_2) \bigr)^2}{\mbox{\rm area}_{\mu_2} (\mathcal{A}_2)} <  
 \frac{L_{\mu'}\bigl (\Gamma(w_0, \mathcal{T}_2) \bigr)^2}{\mbox{\rm area}_{\mu'} (\mathcal{A}_2)}
 \]
 for the metric 
 \[
 \mu' (v) = \begin{cases} \mu_1 (v) & \mbox{ if } v \in V(\mathcal{A}_1) \\
                                           \mu_2 (v) & \mbox{ if } v \in V(\mathcal{A}_2) \setminus  V(\mathcal{A}_1)
                   \end{cases},
 \]
 because ${\mbox{\rm area}_{\mu_2} (\mathcal{A}_2)}= \mbox{\rm area}_{\mu'} (\mathcal{A}_2)$. This contradicts the assumption that $\mu_2$ is an extremal metric
 for $\mbox{\rm VEL}(w_0, \mathcal{T}_2)$, hence we have $L_{\mu_2} \bigl(\Gamma(w_0, \mathcal{T}_1) \bigr) = \mbox{\rm VEL}(w_0, \mathcal{T}_1)^{1/2}$.
 Now the uniqueness part of Theorem~\ref{list}(b)  implies that $\mu_2 |_{\mathcal{A}_1} = \mu_1$. 
 
Suppose $T$ is a component of $\mathcal{A}_2 \setminus \mathcal{A}_1$,  $X$ is the set of vertices in $T$ that have neighbors in $\mathcal{T}_1$, 
and $Y= V(T) \cap \mathcal{T}_2$. 
Note that every geodesic path from $w_0$ to $\mathcal{T}_1$ has the same $\mu_2$-length because $\mu_2|_{\mathcal{A}_1} = \mu_1$. Moreover, we know that 
 at least one $h_2$-triangular mesh is attached to each vertex in $\mathcal{T}_1$ and every component of $\mathcal{A}_2 \setminus \mathcal{A}_1$ is isomorphic to $T$.
 Thus one can check without difficulties that  $\mu_2 |_T$ must be an extremal metric for $\mbox{\rm VEL}(X, Y)$. Then because $T$ could be any component of 
 $\mathcal{A}_2 \setminus \mathcal{A}_1$, we have by Lemma~\ref{lemma1}  that
  \[
 \lambda_2 := \mbox{\rm VEL}(X, Y) = \frac{1}{2} + \frac{1}{3} + \cdots + \frac{1}{h_2+1} \leq \ln (h_2 +1),
 \]
 and $\mu_2 (w) = c/(m+1)$ if $d_{\mathcal{A}_2} (w_0, w) = h_1 + m$, where $c$ is a fixed positive constant and $m \geq 1$. In particular, this implies that 
 $\mu_2 (w) = \mu_2 (w')$ if $w$ and $w'$ have the same combinatorial distance from $w_0$.
 
Let $M= (k_0+6)( 1+d_1 c_1)$, the number of triangular meshes in the second generation, 
 and $\lambda$ be the positive number such that $\lambda^2 = \mbox{\rm VEL}(w_0, \mathcal{T}_1) = 1+ \lambda_1/(k_0 + 6)$. 
 Then $c$ must be  the positive number maximizing 
\[
f(x) = \frac{(\lambda+ \lambda_2 x)^2}{1+ M \cdot \lambda_2 x^2}
\]
for $x>0$, so we obtain $c = 1/M \lambda$ and the maximum is
\begin{align*}
\mbox{\rm VEL}(w_0, \mathcal{T}_2) & = \lambda^2 + \frac{\lambda_2}{M} =  1+ \frac{\lambda_1}{k_0 + 6} + \frac{\lambda_2}{(k_0+6) (1+d_1 c_1)} \\
& \leq 1+ \frac{\ln (h_1 +1) }{k_0 + 6} + \frac{\ln(h_2 +1)}{(k_0+6) d_1 c_1}.
\end{align*}

We repeat the above argument inductively, and note that the number of triangular meshes in the $k$-th generation is given in \eqref{E:generation}.
Thus we  obtain
\begin{equation}\label{vel}
\begin{aligned}
\mbox{\rm VEL}(w_0, \mathcal{T}_n) & = 1+ \frac{\lambda_1}{k_0 + 6} + \frac{1}{k_0 +6} \sum_{k=2}^{n} \frac{\lambda_k}{1+ d_1 c_1 + \cdots + d_{k-1} c_{k-1}} \\
& \leq 1+ \frac{\ln (h_1 +1) }{k_0 + 6} + \frac{1}{k_0 +6} \sum_{k=2}^{n} \frac{\ln (h_k+1)}{d_{k-1} c_{k-1}}
\end{aligned}
\end{equation}
for $n \geq 2$.
 Since $\mbox{\rm VEL}(w_0, \infty)= \lim_{n \to \infty} \mbox{\rm VEL}(w_0, \mathcal{T}_n)$, we conclude that $\mathcal{A}$ is VEL hyperbolic,
hence $G$ is CP hyperbolic, if the series in \eqref{E:series}, or in \eqref{E:series2}, converges. This completes the proof of Theorem~\ref{T:LCP}.

\begin{remark}\label{FR1}
Though we proved the parabolic part of Theorem~\ref{T:LCP} using the Rodin-Sullivan theorem, it is also possible to prove it by computing  $\mbox{\rm VEL}(v_0, \infty)$. 
To see it, let us assume that the series \eqref{E:series} diverges. Then because $1+ d_1 c_1 + \cdots  + d_n c_n \leq 2 d_n c_n$ 
and $2 \lambda_n \geq  \ln (h_n +1)$, where $\lambda_n$ is as before, we see from \eqref{vel} that $\mbox{\rm VEL}(w_0, \mathcal{T}_n) \to \infty$ because
\begin{align*}
\mbox{\rm VEL}(w_0, \mathcal{T}_n) & = 1+ \frac{\lambda_1}{k_0 + 6} + \frac{1}{k_0 +6} \sum_{k=2}^{n} \frac{\lambda_k}{1+ d_1 c_1 + \cdots + d_{k-1} c_{k-1}}\\
& \geq 1+ \frac{\lambda_1}{k_0 + 6} + \frac{1}{4(k_0 +6)} \sum_{k=2}^{n}   \frac{\ln (h_k+1)}{d_{k-1} c_{k-1}}.
\end{align*}
Now for each $n \geq 1$, let $\mu_n$ be the extremal metric for $\mbox{\rm VEL}(w_0, \mathcal{T}_n)$ 
such that ${\mbox{\rm area}_{\mu_n} (\mathcal{A}_n)}=1$, and define $\mu_n (w) = 0$ for $w \in V(\mathcal{A}) \setminus V(\mathcal{A}_n)$.
Choose a subsequence $n_k$ of natural numbers such that $\mbox{\rm VEL}(w_0, \mathcal{T}_{n_k}) \geq 2^k$, and define
$\mu (w) = \sum_{k=1}^\infty \mu_{n_k} (w)/2^{k}$.
Note that $\mu$ is admissible because
\begin{align*}
0 & < \mbox{\rm area}_\mu (\mathcal{A}) = \sum_{w \in V(\mathcal{A})} \mu(w)^2  = \sum_{w \in V(\mathcal{A})} \left( \sum_{k=1}^\infty \frac{\mu_{n_k} (w)}{2^{k}} \right)^2 \\
& \leq \sum_{w \in V(\mathcal{A})}\left( \sum_{k=1}^\infty \frac{1}{2^k} \cdot \sum_{k=1}^\infty \frac{\mu_{n_k} (w)^2 }{2^k} \right) 
= \sum_{w \in V(\mathcal{A})}\left( \sum_{k=1}^\infty \frac{\mu_{n_k} (w)^2 }{2^k} \right) \\
& = \sum_{k=1}^\infty  \sum_{w \in V(\mathcal{A})}\frac{\mu_{n_k} (w)^2 }{2^k} = \sum_{k=1}^\infty \frac{\mbox{\rm area}_{\mu_k} (\mathcal{A}) }{2^k} 
= \sum_{k=1}^\infty  \frac{1}{2^k} = 1.
\end{align*}

We can push down the metric $\mu$ via $\phi : \mathcal{A} \to G$ because we have $\mu(w) = \mu(w')$ if $d_{\mathcal{A}} (w_0, w) =  d_{\mathcal{A}} (w_0, w')$. 
That is, if we  define $\mu' : V(G) \to [0, \infty)$ so that $\mu' \bigl( \phi (w) \bigr) = \mu (w)$, then the metric $\mu'$ is well-defined because 
 $\phi$ is a surjective graph homomorphism satisfying $d_{\mathcal{A}} (w_0, w) = d_{G} (v_0, \phi(w) )$.
Note that $\mu'$ is admissible because $\mbox{\rm area}_{\mu} (\mathcal{A}) \geq  \mbox{\rm area}_{\mu'} (G)>0$, and one can check that
\[
 \infty = L_{\mu}\bigl (\Gamma(w_0, \infty) \bigr)=L_{\mu'}\bigl (\Gamma(v_0, \infty) \bigr).
 \]
Therefore $\mbox{\rm VEL}(v_0, \infty) = \infty$, and we conclude that $G$ is VEL parabolic, hence CP parabolic by Theorem~\ref{list}. 
\end{remark}

\section*{Acknowledgement}
The author thanks to Steffen Rohde for invaluable comments, and Ken Stephenson for developing and making available  the CirclePack software,
by which Figures~\ref{F:cp} and \ref{F:layered} were generated. 
The author especially thanks to the anonymous referee for the detailed and thoughtful review, and highly appreciates to inform him the recurrence condition \eqref{Reci2}
and  the corresponding example in Theorem~\ref{T:S}, which significantly contributed to improving the quality of the paper.
This research was supported by Basic Science Research Program through the National Research Foundation of Korea(NRF) funded by the Ministry of Education(NRF-2017R1D1A1B03034665).

\bibliographystyle{plain}
\bibliography{criteria}

\end{document}